\documentclass{amsart}
\usepackage{amssymb}
\usepackage[usenames]{color}
\newtheorem{ass}{Assumption}

\def\cal{\mathcal}

\newcommand\indiq{{\bf 1}}

\renewcommand\r{\mathbb{R}}
\newcommand\R{\r}
\newcommand\E{\mathbb{E}}
\newcommand\W{\mathcal{W}}

\def\P{{\mathbb P}}

\newcommand\D{D}

\newcommand\eg{g^\varepsilon}
\input{tcilatex}

\begin{document}
\title[Fluctuations for mean field limits of interacting neurons]{Fluctuations for mean field limits of interacting systems of spiking neurons}
\author{Eva L\"ocherbach}
\address{E. L\"ocherbach: SAMM, Statistique, Analyse et Mod\'elisation
Multidisciplinaire, Universit\'e Paris 1 Panth\'eon-Sorbonne, EA 4543 et FR
FP2M 2036 CNRS, France.\\
E-mail: eva.locherbach@univ-paris1.fr}
\maketitle
\begin{abstract}
We consider a system of $N$ neurons, each spiking randomly with rate depending on its membrane potential. When a neuron spikes, its potential is reset to $0$ and all 
other neurons
receive an additional amount $h/N$ of potential, where $ h > 0$ is some fixed parameter. In between successive spikes, each neuron's potential undergoes some leakage at constant rate $ \alpha. $ While the propagation of chaos of the system, as $N \to \infty$, to a limit nonlinear
jumping stochastic differential equation has already been established in a series of papers, see \cite{aaee}, \cite{evafournier}, \cite{evapierre}, the present paper is devoted to the associated central limit theorem. More precisely we study the measure valued process of fluctuations at scale $ N^{-1/2}$ of the empirical measures of the membrane potentials, centered around the associated limit. We show that this fluctuation process, interpreted as c\`adl\`ag process taking values in a suitable weighted Sobolev space, converges in law to a limit process characterized by a system of stochastic differential equations driven by Gaussian white noise. We complete this picture by studying the fluctuations, at scale $ N^{-1/2}, $ of a fixed number of membrane potential processes around their associated limit quantities, giving rise to a mesoscopic approximation of the membrane potentials that take into account the correlations within the finite system.


\textbf{AMS Classification 2010:} 60G55; 60F05: 60G57; 92B20
\end{abstract}

\section{Introduction}
In the present paper we study the fluctuations for the mean field limits of systems of interacting and spiking neurons as the number of neurons tends to infinity. For any fixed size $N,$ the system is characterized by the vector of potential values of the $N$ neurons, $X^N = (X^N_t)_{t \geq 0 }. $ Here, for any time $t \geq 0, $ $X^N_t = (X^{N, 1 }_t, \ldots , X^{N, N}_t )$ and $ X^{N, i }_t \geq 0 $ denotes the membrane potential of neuron $i$ at time $t .$ The process $ X^N $ is a Markov process having generator $L^N$ given by  
\begin{equation}\label{eq:dynintro}
L^N  \varphi ( x) = - \alpha \sum_{i=1}^N \partial_{x^i} \varphi (x) x^i + \sum_{i=1}^N f (x^i)  \left( \varphi ( x + \sum_{j\neq i } \frac{h}{{N}} e_j - x^i e_i  ) - \varphi ( x) \right) ,
\end{equation}
for any smooth test function $ \varphi.$  In the above equation, $ x = ( x^1, \ldots, x^N) \in \r_+^N , $ and $ e_i, 1 \le i \le N, $ denotes the $i-$th unit vector in $ \r^N.$ $h > 0 $ is a positive constant, the synaptic weight, and $ \alpha > 0 $ the leakage time constant. The function $ f : \R_+ \to \R_+$ is the jump rate function. Since $h > 0, $ we are working in the purely excitatory case, such that all membrane potentials take values in $ \r_+.$

The above system of interacting neurons (or slight variations of it) and its mean field limits have been studied in a series of papers, starting with \cite{aaee}, \cite{evafournier} and \cite{robert-touboul}, followed by \cite{tanre1}--\cite{tanre2} which are devoted to the longtime behavior of the associated nonlinear limit process. Spatially structured versions of these convergence results have moreover been obtained in \cite{duarte} and \cite{duartechevallier}. All these papers establish the propagation of chaos property implying that,  in the limit model, different neurons are independent. 
The present paper completes this study by presenting the associated central limit theorem. In particular we will be able to present a mesoscopic approximation for each neuron's potential that takes care of the correlations between different neurons within finite, but large, systems, giving a precise form of the factor of common noise.

\subsection{The model}
To introduce the precise model, consider a family of i.i.d. Poisson measures $(\pi^i(ds, dz ))_{i \geq 1 }$ on $\r_+  \times \r_+$ having intensity measure $ds dz$ each, as well as an i.i.d. family $(X^{i}_0)_{i \geq 1 }$ of $\r_+ $-valued random variables, independent of the Poisson measures, distributed according to some probability measure $ g_0 $ on $ \r.$ 
Then we may represent each neuron's potential as
\begin{multline}\label{eq:dyn}
X^{N, i}_t =  X^{i}_0 - \alpha  \int_0^t  X^{N, i}_s    ds
+ \frac{h}{{N}}\sum_{ j= 1, j \neq i }^N \int_{[0,t]\times\r_+} \indiq_{ \{ z \le  f ( X^{N, j}_{s-}) \}} \pi^j (ds,dz )  \\
- \int_{[0, t ] \times \r_+} X^{N, i }_{s- } \indiq_{\{ z \le f ( X^{N, i }_{s- } ) \}} \pi^i (ds, dz) , 1 \le i \le N.
\end{multline}

It has been shown in \cite{aaee}, \cite{evafournier}  and and \cite{robert-touboul} that under appropriate assumptions on $f$ and $g_0,$ the asymptotic evolution, as $ N \to \infty , $ of the membrane potential processes can be described as solution of the following infinite i.i.d. system of non-linear stochastic differential equations 
\begin{equation}\label{eq:dynlimit}
\bar X^{i}_t =  X^{i}_0 - \alpha  \int_0^t  \bar X^{ i}_s    ds
+ h \int_0^t \E ( f ( \bar X^i_s)) ds  \\
- \int_{[0, t ] \times \r_+} \bar X^{ i }_{s- } \indiq_{\{ z \le f (\bar  X^{ i }_{s- } ) \}} \pi^i (ds, dz) , \, i \geq 1.
\end{equation}

Throughout this paper we strengthen the conditions of \cite{evafournier} and impose the following conditions.

\begin{ass}\label{ass:1}
$ f \in C^6 ( \r_+, \r_+) $ is convex and increasing  such that $ f( x) > 0 $ for all $ x > 0.$ Moreover there exists some $ \alpha \geq 1 $ such that for all $ 0 \le k \le  6 ,$
$$ \sup_{x \geq 1 } [ f' / f + f''/f'] ( x) < \infty \mbox{ and } \sup_{x \in \R_+} \frac{ |f^{(k)}(x)|}{ ( 1 + |x|^\alpha ) } < \infty .$$
\end{ass}
 
Concerning the distribution of the initial potential values we impose
\begin{ass}\label{ass:2}
We suppose that $g_0 $ is compactly supported and possesses a probability density $ g_0 ( x)$ which belongs to $  C^1 ( \r_+, \r_+).$ 
\end{ass}
Then by \cite{aaee} and \cite{evafournier}, there exists a unique strong solution both for \eqref{eq:dyn} and for \eqref{eq:dynlimit}. Moreover,  constructing $ X^{N, i } $ and $ \bar X^i $ using the same underlying Poisson random measure $ \pi^i ,$ we have for any $ T> 0$ and for any $ 1 \le i \le N, $  
\begin{equation}\label{eq:prop1}
\sup_{ t \le T} \E ( | X_t^{N, i } - \bar X_t^i | ) \le \frac{C_T}{\sqrt{N}} ,
\end{equation}
where the constant $C_T $ does not depend on $N.$ Introducing 
\begin{equation}\label{eq:gt}
g_t = {\mathcal L} ( \bar X^1_t) , 
\end{equation}
and  the empirical measure of the finite system together with the associated projection onto time $t, $
\begin{equation}\label{eq:muN}
 \mu^N = \frac1N \sum_{ i=1}^N \delta_{X^{N, i } }, \;  \mu^N_t = \frac1N \sum_{ i=1}^N \delta_{X^{N, i }_t }, 
\end{equation}  
we also have 
\begin{equation}\label{eq:prop2}
\sup_{ t \le T} \E ( \W_1 ( \mu^N_t, g_t ) ) \le \frac{C_T}{\sqrt{N}}.
\end{equation}
Here, the Monge-Kantorovich-Wasserstein distance $ \W_1 ( \mu , \nu) $ between two probability measures $\mu$ and 
$\nu$ on $\R_+$ with finite expectations is defined by $\W_1(\mu,\nu)=\inf\{\E[|U-V|]$, ${\mathcal L}(U)=\mu$ and 
${\mathcal L}(V)=\nu\}$.

Finally, let us mention that $ g_t = {\mathcal L} ( \bar X^i_t) $ is solution of a nonlinear PDE which in its strong form reads as 
$$ \partial_t g_t ( x) = [ \alpha x - h p_t ] \partial_x g_t(x) + ( \alpha - f(x) ) g_t ( x), t \geq 0, x > 0 , \; p_t = \int_0^\infty f ( x) g_t (x) dx,  $$
starting from the initial value $g_0, $ together with the boundary condition 
$ g_t (0)  = 1/h $ for all $t > 0.$  

As a consequence of \eqref{eq:prop2}, interpreting $ \mu^N_t$ as random variable in the space $ \cal P ( \r_+) $ of all probability measures on $ \r_+,$  we have convergence in probability 
$  \mu^N_t \to g_t ,$
as $ N \to \infty, $ and the rate of convergence is at least $N^{-1/2}. $ It is therefore natural to study the associated process of fluctuations, given by 
\begin{equation}\label{eq:etan}
\eta_t^N = \sqrt{N} ( \mu_t^N - g_t ) ,
\end{equation}
together with the fluctuations of the processes of membrane potentials 
\begin{equation}\label{eq:utnk}
U_t^{N, i } = \sqrt{N} ( X_t^{N, i } - \bar X^i_t) , 1 \le i \le n ,
\end{equation}
for a fixed number of $n$ neurons, where for each $ 1 \le i \le n, $ $ X^{N, i } $ and $ \bar X^i $ are constructed according to the so-called Sznitman coupling (see \cite{sznitman}): They are defined on the same probability space, starting from the same initial condition $X^i_0$ and using the same underlying Poisson random measure $ \pi^i ,$ for each $ 1 \le i \le n.$  

In the present paper we prove convergence in law of the sequence of processes $ (( U^{N, 1 }, \ldots , U^{N, n } ) , \eta^N) $ to a limit process $ (( \bar U^{1 }, \ldots , \bar U^{n } ) , \bar \eta) ,$ as $ N \to \infty ,$ for any fixed $n.$  The limit process $\bar \eta ,$ interpreted as distribution acting on appropriate test functions,  follows an infinite dimensional differential equation stated precisely in \eqref{eq:bareta} below. Moreover, for each $ 1 \le i \le n, $ the limit process $ \bar U^i $ follows an Ornstein-Uhlenbeck dynamic with variable length memory, that is, for any $ t \geq 0, $ 
\begin{equation}\label{eq:barui}
  \bar U^i_t = - \int_0^t \alpha  \bar U^i_s ds +h \int_0^t \bar \eta_s (f) ds   -\int_{[0, t] \times \r_+} \bar U^i_{s-} \indiq_{\{ z \le f ( \bar X^{i }_{s-} ) \}} \pi^i  ( ds, dz) + h M_t  .
\end{equation}
Here, $ (M_t )_t $ is a Gaussian martingale having quadratic variation 
$$ <M>_t = \int_0^t g_s ( f) ds  = \E \int_0^t f ( \bar X^i_s) ds .$$ 
In \eqref{eq:barui}, the presence both of this Gaussian martingale and of the integral of the fluctuations of the spiking rate, $ \int_0^\cdot \bar \eta_s ( f) ds, $ induces a factor of common noise explaining the correlations between different neurons in the finite system.  

As a consequence, we obtain the following second order error correction to the mean field approximation 
\begin{equation}\label{eq:meso}
 X_t^{N, i } = \bar X_t^i + \frac{1}{\sqrt{N}} \bar U_t^i , \mbox{ where  }  \; \bar U^i_t = h \int_{\bar L_t^i }^t e^{ - \alpha ( t-s) }\bar \eta_s (f) ds  + h \int_{\bar L_t^i }^t e^{ - \alpha ( t-s) } d M_s ,  
\end{equation} 
with $ L_t^i = \sup \{ s \le t: \Delta \bar X^i_s \neq 0 \} $ the last spiking time of neuron $i$ in the limit process, before time $t,$ with $ \sup \emptyset := 0.$  
 
While in \eqref{eq:barui} above, the convergence of $ (U^{N, 1 }, \ldots , U^{N, n } )$ has to be understood as convergence of stochastic processes with c\`adl\`ag trajectories, that is, of random variables taking values in $D ( \r_+, \R^n ), $ we did not specify so far in which space the convergence of the rescaled empirical measures $\eta^N$ takes place. Following the Hilbertian approach introduced in  \cite{ferland} and \cite{fernandez} and then applied to the framework of point processes in \cite{chevallier}, throughout this paper  we interpret $ \eta^N$ as stochastic process taking values in a suitable distributional space which is the dual of some weighted Sobolev space of test functions. 
The regularity of test functions we need to impose is related to the order up to which we have to develop the error terms that appear when replacing the contribution of small jumps (i.e., the last term appearing in \eqref{eq:dynintro}) by the associated limit drift. 
Moreover, since the finite size process does not take values in a compact set, we need to work with a Sobolev space supported by $ \R_+.$ Finally, it turns out that we have to include constant functions into our class of admissible test functions, as well as  the firing rate function $f$ which is of polynomial growth. Therefore we are led to work with weighted Sobolev spaces, where the weights are chosen to be polynomial, of power $ p > \alpha + \frac12  ,$ where $\alpha $ is the growth rate of $f$ and its derivatives (see Assumption \ref{ass:1}).

The approach used in this article follows closely the study of fluctuations for McKean-Vlasov diffusions in \cite{fernandez} and the adaptation of this work to the framework of age-dependent Hawkes process proposed in \cite{chevallier}. The main difference with respect to \cite{fernandez} is that, as in \cite{chevallier}, the limit processes $ \bar X^i $ and $ \bar U^i $ remain jump processes; the big jumps induced by spikes survive also in the limit process. 
The main difference with respect to \cite{chevallier} is the following. Being interested in age-dependent Hawkes processes, in \cite{chevallier}, the limit process undergoes a deterministic drift given by $ b (x ) = 1 . $ This trivially implies good coupling properties. 
In our model however, the time dependent drift of the limit process is given by $ - \alpha x + h g_t ( f)  $ at time $t$ and depends both on the position $x,$ but also on the average spiking rate of the system. This makes the study of coupling more complicated, which is one of the main reasons why it is more difficult to prove the uniqueness of the limit equation in the present frame. In particular, to prove the uniqueness, we do also have to establish regularity properties of the time inhomogeneous semigroup  associated to the limit process \eqref{eq:dynlimit} which is non-diffusive and associated to a transport equation. We rely on Girsanov's theorem for jump processes to tackle this problem, see Proposition \ref{prop:Pst} below.

\subsection{General notation.} 
The space of bounded functions of class $ C^k, $ defined on $ \r_+,$ with bounded derivatives of each order up to order $k,$ is denoted by $ C_b^k.$ 
$C^\infty_c $ denotes the space of infinitely differentiable functions defined on $ \R_+,$ having compact support. 
The space of c\`adl\`ag functions defined on $ \r_+$ and taking values in some Polish space $E$ is denoted by $ D ( \r_+, E).$ 
If $ \mu $ is a measure on $E$ and $ \varphi : E \to \r $ measurable and integrable, we write $ < \mu, \varphi > := \int_E \varphi d \mu.$

Throughout this paper we work with the canonical filtration $( {\mathcal F}_t)_{t \geq 0} $ where $ {\mathcal F}_0= \sigma \{ X^i_0, i \geq 1 \} $ and $ {\mathcal F}_t =  \sigma \{ X^i_0, i \geq 1 , \pi^j ( A) : A \subset [0, t ] \times \r_+, j \geq 1\}.$ 

Finally, $C$ denotes a constant that may change from one occurence to another, even within one line, and for the jump rate function $f$ introduced above and any $x = ( x^1, \ldots , x^N)  \in \R^N, N \geq 1, $ we shall write 
\begin{equation}\label{eq:barf}
\bar f ( x) = \sum_{i=1}^N f ( x^i ).
\end{equation}

\section{Main results}
We fix some $ n \geq 1.$ The aim of this section is to state the convergence in law of the sequence of processes $ ((U^{N,i})_{1 \le i \le n}, \eta^N )_N $ defined by 
\begin{equation}\label{eq:un}
U^{N, i }_t := \sqrt{N} ( X^{N, i }_t - \bar X^i_t) \mbox{ and }  \eta^N_t = \sqrt{N} ( \mu^N_t - g_t ),
\end{equation}
where we interpret $\eta^N$  as stochastic process with values in a suitable space of distributions. In the above definition, $ X^{N, i }$ and $\bar X^i$ are 
constructed according to the Sznitman coupling (see \cite{sznitman}), that is, using the same initial value $ X_0^i$ and driven by the same underlying Poisson random measure $ \pi^i.$

We start gathering some basic definitions and results on weighted Sobolev spaces. 

\subsection{Weighted Sobolev spaces}
Since we are working in the purely excitatory case and the membrane potentials take values in $ \R_+ ,$ in what follows, all test functions that we consider are defined on $ \R_+.$ Fixing an integer $k$ and a positive real number  $ p \geq 0,  $  we introduce the norm $ \| \psi \|_{k, p } $  for all functions $ \psi \in C^\infty_c$ given by
$$ \| \psi \|_{k, p } := \left( \sum_{l=0}^k \int_0^\infty \frac{ | \psi^{(l)} ( x)|^2 }{ 1 + |x|^{2p}} dx \right)^{1/2} $$
and define the space $ \W^{k, p }_0 $ to be the completion of $ C^\infty_c $ with respect to this norm. 

The space $ \W^{k, p }$ is a separable Hilbert space, and we denote $ \W^{ - k, p }_0$ its dual space, equipped with the norm $ \| \cdot \|_{ - k , p } $ defined for any $ \eta \in \W_0^{ - k, p }$ by 
$$ \| \eta \|_{ - k ,  p } = \sup \{ | < \eta , \psi >| :  \psi \in \W_0^{k, p }, \| \psi \|_{k, p } = 1  \}.$$ 
Finally, $ C^{ k,  p}$ is the space of all  $ C^k -$functions such that for all $ l \le k, $ $ \sup_{x \in \R_+} | \psi^{(l)} ( x) | / ( 1 + |x|^p ) < \infty .$ This space is equipped with the norm 
$$ \| \psi \|_{ C^{k, p } } := \sum_{l=0}^k \sup_{x \in \R_+} \frac{ | \psi^{(l)} (x) |}{1 + |x|^p } < \infty .$$

The most important facts about Sobolev spaces that we use throughout this paper are collected in the Appendix Section \ref{sec:sobolev}.

\subsection{Weak convergence of the fluctuation process}
Given the knowledge of the function $ t \mapsto g_t (f) , $ the non-Markovian limit process \eqref{eq:dynlimit} is described by the time dependent infinitesimal generator given by
\begin{equation}\label{eq:generator}
 L_s \varphi ( x) = - \alpha x \varphi ' (x) + g_s ( f) \varphi ' (x) + f (x) R \varphi ( x) , \quad R \varphi ( x) = \varphi(0) - \varphi ( x), 
\end{equation} 
for all $ s \geq 0$ and $ \varphi \in C^1_b ,$ where $g_s$ is given by \eqref{eq:gt}. Our main result reads as follows. 

\begin{theorem}\label{theo:main}
Fix some $n \geq 1.$ Under Assumptions \ref{ass:1} and \ref{ass:2}, for any $ p > \alpha + \frac12  , $ we have convergence in law of $( (U^{N, 1} , \ldots, U^{N, n } ), \eta^N) $ in $ D ( \R_+ ,\R^n \times  \W_0^{ - 4, p } ) $ to the limit process $ ((\bar U^1, \ldots, \bar U^n ), \bar \eta) $ taking values in $ D ( \R_+ ,\R^n) \times C ( \R_+, \W_0^{- 4, p } ),$ which is solution of the system of stochastic differential equations
\begin{equation}\label{eq:baru}
  \bar U^i_t = - \alpha  \int_0^t \bar U^i_s ds +h \int_0^t \bar \eta_s (f) ds   -\int_{[0, t] \times \r_+} \bar U^i_{s-} \indiq_{\{ z \le f ( \bar X^{i }_{s-} ) \}} \pi^i ( ds, dz) + h W_t ( 1) , 1 \le i \le n , 
\end{equation}
and
\begin{equation}\label{eq:bareta}
\bar  \eta_t ( \varphi) = \bar \eta_0 (\varphi) + \int_0^t \bar \eta_s ( L_s \varphi ) ds +h \int_0^t g_s ( \varphi ') \bar \eta_s ( f) + W_t( R \varphi )+h  \int_0^t  g_s ( \varphi ')  d W_s ( 1)  , 
\end{equation}
where the above equation holds for all  $ \varphi \in \W_0^{ 5, p }, $ 
and where, for any $ \psi \in \W_0^{4, p }, $ 
\begin{equation}\label{eq:wt}
 W_t ( \psi) = \int_0^t \int_\R \sqrt{ f (x) } \psi ( x) d M (s, x) ,
\end{equation} 
with
$M (dt, dx) $  an orthogonal martingale measure on $ \r_+ \times \r $ with intensity  $dt  g_t(dx).$ 

If we suppose moreover that there exist (possibly small)  $c_1,  \beta > 0 $ and (possibly large) $K > 0 $ such that for all $ x \geq K,$ $ f(x) \geq c_1 x^\beta ,$ then the limit process solving \eqref{eq:baru} and \eqref{eq:bareta} is unique. 
\end{theorem}

\begin{remark}
1. The assumption on the minimal growth rate $ x^\beta $ of the rate function $f(x) $ for large values of $x$ can be dropped if $f \in C^6_b.$ 

2.Notice that by construction,
$$ < W ( \varphi), W( \psi) >_t = \int_0^t \int_\R g_s (dx) \varphi ( x) \psi (x) f(x) ds = \int_0^t g_s ( f \varphi \psi) ds = \E \int_0^t( f \varphi \psi) (\bar X^i_s) ds  .$$ 
\end{remark}

\subsection{Plan of the paper}
The remainder of this paper is devoted to the proof of Theorem \ref{theo:main}. Section \ref{sec:3} starts  with useful a priori bounds on the finite size process and its limit, before establishing the uniqueness of any solution of \eqref{eq:baru} and \eqref{eq:bareta} in Theorem \ref{theo:uniqueness}. We then continue, in Section \ref{sec:3bis} by establishing a decomposition of the finite size fluctuations  in Proposition \ref{prop:14} which is the starting point of the proof of our main result. We prove tightness of $ ( \eta^N, U^N) $ in Theorem \ref{theo:tight} of Section \ref{sec:4}. Theorem \ref{theo:char} in Section \ref{sec:5} then states that any possible limit $ (\bar \eta , \bar U) $ of $ ( \eta^N, U^N) $ is necessarily solution of the system of differential equations of Theorem \ref{theo:main}. The Appendix section collects some useful results about the limit process together with some technical results.

\section{Uniqueness of the limit equation}\label{sec:3}
\subsection{Preliminaries}
We first investigate the mappings that appear in the generator of the limit process. These are the linear mapping $ R $ defined by $ R \varphi := \varphi( 0) - \varphi ,$ the mapping 
$ xD : \varphi \mapsto [ x \mapsto x \varphi ' (x)]$ and the mapping $ D : \varphi \mapsto \varphi'  .$

\begin{lemma}\label{prop:5}
$ R$ is a continuous mapping from $ \W^{k,p }_0 $ to itself, for any $ k \geq 1 $ and $ p > \frac12. $ If we suppose moreover that Assumption \ref{ass:1} holds, then for any $ p > \alpha + \frac12 , $ and any $ k \le 6,$ 
$$ \| f R \varphi \|_{ k, p } \le C \|f\|_{C^{k, \alpha} } \| R \varphi \|_{ k, p- \alpha } \le C   \|f\|_{C^{k, \alpha} } \|\varphi \|_{ k, p- \alpha }.$$
Finally, for any $ k \geq 2, $ $D:   \W_0^{k, p } \to   \W_0^{k-1, p}$ and 
$ xD  :   \W_0^{k, p } \to   \W_0^{k-1, p+1}$ are continuous mappings satisfying $ \| D \varphi \|_{k-1, p } \le C \| \varphi \|_{k, p } $ and $ \| xD \varphi \|_{  k-1, p+1 } \le C \|\varphi\|_{k, p } .$

As a consequence and since $ \alpha \geq 1,$ the application  $L_s$ introduced in \eqref{eq:generator}  is a linear continuous mapping from $ \W_0^{ k, p } $ to $ \W_0^{ k-1, p+\alpha } ,$ for any $ p > \frac12 , $  $ k \leq 6,$ and for all $ \psi \in \W_0^{k, p }, $ 
$$ \sup_{ s \le T } \frac{ \| L_s \psi\|_{k-1, p+\alpha}^2}{\|\psi \|^2_{k, p } } < \infty .$$ 
\end{lemma}
The proof of the above lemma is immediate. In the sequel we shall also rely on the following result.

\begin{lemma}\label{lem:6}
For any $ x \in \R_+$ and any $p > 0, $  the mapping  $\delta_x :  \W_0^{1, p } \to \r, \psi \mapsto \psi ( x) $  is continuous. Moreover there exists a constant $C$ not depending on $x $ such that  
\begin{equation}\label{eq:deltax}
 \| \delta_x\|_{- 1, p } \le C ( 1 + |x|^p ).
\end{equation}  
Similarly, $ D^* \delta_ x   :  \W_0^{2, p } \to \r, \psi \mapsto \psi '( x) $  is continuous and there exists a constant $C$ not depending on $x $ such that  
$$ \| D^* \delta_x \|_{- 2, p } \le C ( 1 + |x|^p ).$$ 
\end{lemma}

\begin{proof}
We only show the second assertion. We have for any $ \psi \in \W_0^{2, p}, $ 
$$ | <D^* \delta_x  ,  \psi >  | = | \psi' ( x) | \le \| \psi\|_{C^{1, p } } ( 1 + |x|^p ).$$
Moreover, using the Sobolev embedding, there exists a constant $C$ not depending on $x,$ such that 
$ \| \psi\|_{C^{1, p } }  \le C \| \psi\|_{2, p}.$
This implies the assertion.
\end{proof}

The following a priori bounds on \eqref{eq:dyn} and \eqref{eq:dynlimit} will be used throughout this paper. 
\begin{lemma}\label{lemma:apriori}
Under Assumptions \ref{ass:1} and \ref{ass:2}, for any $ 1 \le i \le N $ and $ N \geq 1, $ there exists a constant $c_0$ only depending on $g_0$ 
such that
\begin{equation}\label{eq:poisson}
 X^{N,i}_t \leq  c_0 + 4 h N^N_t,
\end{equation}
where $N^N_t:= N^{-1} \sum_{j = 1 }^N \int_{[0, t ] \times \R_+} 
\indiq_{\{ z \le f(2h) \}} \pi^j ( ds, dz ).$
In particular, for any $ T, p > 0,$
\begin{equation}\label{eq:aprioribound1}
 \sup_N \E ( \sup_{ t \le T } | X_t^{N, i }|^p ) \le C_T (p)  
\end{equation} 
for a constant depending only on $T,$ $p$ and $g_0.$ Introducing the set 
\begin{equation}\label{eq:gtn}
 G_T^N = \left\{ \sum_{ i=1}^N \int_{ [0, T ] \times \R_+} \indiq_{\{ z \le f ( 2h ) \}} \pi^i ( ds, dz) \le 2 f(2h)  N T \right\} ,
\end{equation} 
we also have the upper bound 
\begin{equation}\label{eq:aprioriboundgt}
\indiq_{ G_T^N} \left( \sup_{t \le T } \sup_{1 \le i \le N}  |X_t^{N, i } |\right)  \le c_0 + 8 h f(2h) T  \mbox{ and the control } \mathbb{P} ( (G_T^N)^c) \le a e^{ -b N T},
\end{equation}
for some constants $a,b.$ Finally there exists a constant $\bar C_T$ only depending on $g_0, $ such that 
\begin{equation}\label{eq:aprioribound2}
 \sup_{ t \le T } | \bar X_t^{ i }| \le \bar C_T . 
\end{equation} 
\end{lemma}

The proof of this lemma is given in the Appendix. We immediately state a useful corollary. Introducing 
\begin{equation}\label{eq:ztn}
Z_t^{N, i } = \int_{ [0, t ] \times \r_+} \indiq_{\{ z \le f ( X_{s-}^{N, i } ) \}} \pi^i ( ds, dz ) \mbox{ and } \bar Z_t^i = \int_{ [0, t ] \times \r_+} \indiq_{\{ z \le f ( \bar X_{s-}^{ i } ) \}} \pi^i ( ds, dz )
\end{equation}
and  the total variation distance 
$$ \| Z^{N, i } - \bar Z^i \|_{TV, [0, T]} := \# \{ t \le T :  t \mbox{ is a jump of } Z^{N, i } \mbox{ or of } \bar Z^i\mbox{ but not of both} \} ,$$ 
we have that 
\begin{corollary}\label{cor:tvdistance}
Under Assumptions \ref{ass:1} and \ref{ass:2},
\begin{equation}\label{eq:tv}
\E \| Z^{N, i } - \bar Z^i \|_{TV, [0, T]} = \E \int_0^T | f ( X^{N,  i }_s ) - f ( \bar X^i_s) | ds  \le C_T N^{- 1/2},
\end{equation}
for a constant $C_T$ only depending on $T,$ but not on $N.$ 
\end{corollary}

\begin{proof}
Clearly,  
\begin{eqnarray*}
&&\sqrt{N} \E \| Z^{N, i } - \bar Z^i \|_{TV, [0, T]}\\
 & =& \sqrt{N} \E \int_{ [0, T ] \times \R_+} \left| \indiq_{\{ z \le f ( X^{N, i}_{s-} )  \}} - \indiq_{\{ z \le f( \bar X^i_{s-} ) \}} \right|  \pi^i ( ds, dz ) 
=\sqrt{N} \E \int_0^T | f ( X^{N, i}_{s} ) -  f( \bar X^i_{s} )| ds \\
 &\le& \sqrt{N} \int_0^T \E [ | f ( X^{N, i}_{s} ) -  f( \bar X^i_{s} )|, G_T^N]   ds 
+  \sqrt{N} T f ( \bar C_T) \P ( (G_T^N)^c) \\
 &&+   C \sqrt{N} T \E  \left(   \indiq_{ (G_T^N)^c}  [ 1 + (c_0 + 4h N_T^N)^\alpha ]\right) ,   
\end{eqnarray*}
where we have used that by \eqref{eq:poisson} and since $f$ is non-decreasing, 
$$ f(  X^{N,i}_{s} ) \le C ( 1 + (c_0 + 4h N_T^N)^\alpha ) \mbox{ and } f ( \bar X^i_s) \le f( \bar C_T ) $$
for all $ s \le T.$ 

Due to \eqref{eq:aprioriboundgt}, on $ G_T^N, $ $ X_t^{N, i } \le c_0 + 8 h f(2h) T $ for all $ i , $ and all $t \le T.$ Therefore, 
using the Lipschitz continuity of $f$ on $[0, c_0 + 8 h f(2h) T)\vee \bar C_T  ] $ and \eqref{eq:prop1}, we have 
$$ \sup_N \sqrt{N} \int_0^T \E [ | f ( X^{N, i}_{s-} ) -  f( \bar X^i_{s-} )|, G_T^N]   ds \le C_T .$$ 
Using the deviation estimate on $ \P ( (G_T^N)^c) $ together with H\"older's inequality implies moreover that 
$$ \sup_N  \left(\sqrt{N} T f ( \bar C_T) \P ( (G_T^N)^c) 
+   C \sqrt{N} T \E  \left(   \indiq_{ (G_T^N)^c}  [ 1 + (c_0 + 4h N_T^N)^\alpha ]\right) \right) \le C_T $$
such that 
$$ \sup_N \sqrt{N} \E \| Z^{N, i } - \bar Z^i \|_{TV, [0, T]} = \sup_N \sqrt{N} \E \int_0^T | f ( X^{N, i}_{s} ) -  f( \bar X^i_{s} )| ds  \le C_T < \infty ,$$ 
implying the assertion.
\end{proof}
After these preliminary results, we now turn to the proof of our first main result which is the uniqueness of the limit equation.
 
\subsection{Uniqueness}\label{sec:6} 
This section is devoted to the proof of the following 
\begin{theorem}\label{theo:uniqueness}
Grant Assumptions \ref{ass:1} and \ref{ass:2} and suppose moreover that there exist (possibly small)   $c_1,  \beta > 0 $ and $K > 0 $ such that for all $ x \geq K,$ $ f(x) \geq c_1 x^\beta .$ Then for any fixed initial condition $((\bar U^i_0)_{1 \le i \le n } , \bar \eta_0) $ and driving underlying noise $ \pi^i, 1 \le i \le n, $ and $ W,$ the system \eqref{eq:baru}--\eqref{eq:bareta} has at most one solution in  $ D ( \R_+, \R^n ) \times C ( \R_+, \W_0^{- 4, p } ) ,$ for any $p >  \alpha + \frac12  . $ 
\end{theorem}
Since given $(\bar X^i)_{ 1 \le i \le n},$ $ \bar \eta $ and $W, $ the equation for $ \bar U^i , 1 \le i \le n, $ is linear, it is sufficient to prove uniqueness for $ \bar \eta.$

Suppose $ \bar \eta $ and $ \hat \eta $ are both solution of \eqref{eq:bareta}, driven by the same underlying $W$ and starting from the same initial condition. Then $ \tilde \eta_t := \bar \eta_t - \hat \eta_t $ satisfies 
\begin{equation}\label{eq:tildeeta}
<\tilde \eta_t , \varphi > = \int_0^t < \tilde \eta_s , L_s \varphi > ds  + h  \int_0^t g_s ( \varphi ' )\tilde \eta_s ( f) ds ,
\end{equation}
where we recall that
$$ L_s \varphi (x) = - \alpha x \varphi ' (x) + h g_s (f) \varphi ' (x) + f (x) R \varphi (x)  .$$ 
Traditionally, to prove uniqueness we have to deduce from \eqref{eq:tildeeta} that $ \tilde \eta = 0 ,$ that is, $\| \tilde \eta_t \|_{- k, p } = 0 $ for suitable $k$ and $p.$  However, when applying $ \| \cdot \|_{- k , p } $ to \eqref{eq:tildeeta}, we have to treat the term 
$\int_0^t < \tilde \eta_s , L_s \varphi > ds $ which involves a derivative and multiplication with $f$ and therefore gives rise to $ \| \tilde \eta_s\|_{-k +1, p+ \alpha  } $ which cannot be compared to the norm $ \| \tilde \eta_s\|_{-k, p}$ since it is greater. The same problem arises when treating the last term 
$\int_0^t g_s ( \varphi ' )\tilde \eta_s ( f) ds.$ 

Of course, this problem has already appeared -- and solved -- both in \cite{fernandez} and \cite{chevallier}, yet in a simpler framework, since in \cite{fernandez}, the underlying diffusion gives generates regularity of the associated semigroup, while in \cite{chevallier} the underlying flow is particularly simple, having drift $ \equiv 1 .$  
In what follows we show how to adapt these ideas to the present frame and propose severals tricks to get rid of the above derivatives by using integration by parts or by solving directly the flow associated to $ L_s.$ 

We start  gathering known results about the marginal law $g_s$ of the limit process $ \bar X^1_s$ of \eqref{eq:dynlimit}, starting from $ \bar X^1_0 \sim g_0 ,$ that is,  $ g_s = {\mathcal L} ( \bar X^1_s).$ We introduce the associated flow 
\begin{equation}\label{eq:limitflow}
 \varphi_{s, t } (x) = e^{ - \alpha ( t-s)} x +h  \int_s^t e^{ - \alpha (t-u)} g_u (f) du ,
\end{equation} 
representing the evolution of $ \bar X^1 $ in between the successive jumps. 

\begin{proposition}\label{prop:proprietegt}
Under Assumptions \ref{ass:1} and \ref{ass:2}, for all $ t \geq 0, $ $g_t (dy) = g_t (y) dy $ is absolutely continuous having Lebesgue density $g_t (y) ,$ and for all $ t \le T, $ $g_t $ is compactly supported, that is, $ g_t ( y) = 0 $ for all $ y \geq \bar C_T, $ where $\bar C_T$ is as in \eqref{eq:aprioribound2}. Moreover, for all $ y \neq \varphi_{0, t } (0) , $ $g_t$ is differentiable in $y$ having derivative $g_t' $ which is continuous on $ (0, \varphi_{0, t} (0 ) ) \cup (\varphi_{0, t} (0 ), \infty ).$ Finally, 
$$ t \mapsto \int_0^\infty (1 +| x|^p ) |g_s' (x) | dx \mbox{ is locally bounded,}$$
for all $ p \geq 0.$ 
\end{proposition}
The proof of the above result is postponed to the Appendix.

Using integration by parts this implies that we may rewrite the last term appearing in \eqref{eq:tildeeta} as follows.
\begin{eqnarray*}
 g_s( \varphi ' ) &=& \int_0^{\varphi_{0,s } (0) }  \varphi' ( x) g_s ( x) dx + \int_{\varphi_{0,s } (0)}^\infty \varphi' ( x) g_s ( x) dx  \\
& =& g_s ( \varphi_{0,s } (0) - ) \varphi (\varphi_{0,s } (0))  -  g_s(0)  \varphi ( 0 ) - g_s ( \varphi_{0,s } (0) + ) \varphi (\varphi_{0,s } (0)) - \int_0^\infty \varphi ( x) g_s' (x) dx\\
&=&  - \frac{1}{h}  \varphi ( 0 ) - \Delta g_s (  \varphi_{0,s } (0)) \varphi (  \varphi_{0,s } (0)) - \int_0^\infty \varphi ( x) g_s' (x) dx ,
\end{eqnarray*} 
for any $s > 0, $ where we used the identity $g_s ( 0 ) = \frac1h $ which follows from \eqref{eq:density1} stated in the Appendix below, and where, using further \eqref{eq:density2}, 
$$ \Delta g_s (  \varphi_{0,s } (0)) = g_s ( \varphi_{0,s } (0) + ) - g_s ( \varphi_{0,s } (0) - ) = e^{ - \int_0^t ( f ( \varphi_{0, u }(0) ) - \alpha) du } [ g_0 ( 0) - \frac1h ] ,$$
such that 
$$   h \int_0^t g_s ( \varphi ' ) \tilde \eta_s ( f) ds = \int_0^t h_s ( \varphi ) ds  , $$ 
where 
$$ h_s (\varphi ) = - \varphi  ( 0 ) \tilde \eta_s( f) - h \varphi (  \varphi_{0,s } (0)) \Delta g_s (  \varphi_{0,s } (0)) \tilde \eta_s( f)   - h (\int_0^\infty  \varphi  (x)  g_s' (x) dx ) \tilde \eta_s ( f) .$$ 
Relying on Proposition \ref{prop:proprietegt}, we deduce
\begin{proposition}\label{prop:hs}
Let $ \psi \in \W_0^{k, q},  $ for some $k \geq 1,  q \geq 0.$  Fix $ T > 0.$ Then for all  $0 \le t \le T ,  $ and for all  $ p > \alpha + \frac12,$ 
$$ | h_t (\psi) | \le C_T \| \psi \|_{k, q}  \| \tilde \eta_t\|_{- 4, p  } \|f\|_{4, p}. $$
\end{proposition}

\begin{proof}
We use that by the Sobolev embedding,
$$ | \psi ( x) | \le \| \psi\|_{C^{0, q } } (1 + |x|^q ) \le C \| \psi\|_{k, q} ( 1+ |x|^q) ,$$
since $ k \geq 1, $ such that 
$$  | \psi ( 0 ) | \le C \| \psi\|_{k, q} \mbox{ and } | \psi ( \varphi_{0,t } (0)) \Delta g_t (  \varphi_{0,t } (0))  | \le C_T   \| \psi\|_{k, q} $$
and moreover 
$$ \int_0^\infty |\psi ( x) g_s' (x) | dx \le C \| \psi\|_{k, q } \int_0^\infty ( 1 + |x|^q ) |g_s' (x) | dx \le C_T | \psi\|_{k, q } ,$$
where we have used the bound of Proposition \ref{prop:proprietegt}. The conclusion follows from 
$$ | \tilde \eta_t ( f) | \le \| \tilde \eta_t\|_{-4,p } \| f \|_{4, p } , $$
since $ f \in \W_0^{4, p } $  by Assumption \ref{ass:1}, due to the fact that $ C^{4, \alpha } \subset \W_0^{4, p }$ for any $ p > \alpha + \frac12 .$  
\end{proof}

We now turn to the study of the action of $ L_s.$ Given the fixed function $ t \mapsto g_t(f) $, $t \geq 0, $ we introduce the time inhomogeneous Markov process $Y_{s, t } ( x) , $ for any $ 0 \le s \le t $ and $x \in \R_+,$ which is solution of 
$$ Y_{s, t } ( x) =x + \int_s^t \left( h g_u (f) - \alpha Y_{s, u } (x)) \right) du - \int_{]s, t ] \times \R_+} Y_{s, u-} (x) \indiq_{\{ z \le f ( Y_{s, u-} (x) ) \}} \pi^1 (du, dz ) .$$
Clearly, since $ h \int_0^t g_s ( f) ds \le \bar C_T $ for all $ t \le T, $ by \eqref{eq:aprioribound2}, $ Y_{s, t } ( x) \le x + \bar C_T, $ for all $ s \le t \le T,$ such that the above process is well-defined.  
We denote $ P_{s, t } $ the associated semigroup, that is, $ P_{s, t } \psi (x) = \E \psi (  Y_{s, t } ( x) ) ,$ for any measurable test function.  

\begin{proposition}\label{prop:Pst}
Under the assumptions of Theorem \ref{theo:uniqueness}, we have that for any $ 0 \le s \le t $ and any $ p \geq 0, $  $P_{s, t } $ is a continuous mapping from $\W_0^{6, p } \to  \W_0^{6, p} ,$ and 
$$ \| P_{s, t } \psi \|_{k, p } \le  C_T \| \psi\|_{k, p } , $$
for all $ k \le 6, $ $ s \le t \le T.$ 

Moreover, for any $\psi \in C^\infty_c, $ $P_{s, t } \psi $ belongs to $ C^6_b $ and is rapidly decreasing, that is, for all $ \gamma > 0 $ and all $ k \le 6, $ 
\begin{equation}\label{eq:expdecrease}
 \lim_{ x \to \infty } x^\gamma | (P_{s, t } \psi )^{(k)} )(x)| = 0.
\end{equation} 
\end{proposition}
The proof of this result is also postponed to the Appendix. 

We notice that $L_s$ is the time dependent infinitesimal generator associated to the time inhomogeneous semigroup $ P_{s, t }, $ that is, 
$$ \frac{d}{ds} P_{s, t } \psi = - L_s P_{s, t } \psi \mbox{ and } \frac{d}{dt} P_{s, t } \psi = P_{s, t }L_t \psi , $$ 
whenever the above quantities are well-defined.

Now we proceed further with our proof. Let $0 \le s \le t  \le T $  be fixed. Consider a test function $ \psi \in C^\infty_c.$ Then we have that 
$$P_{s, t }  \psi (x)  = \psi ( x) -  \int_s^t \frac{\partial}{\partial v } P_{v, t} \psi (x) ) dv  = \psi ( x) + \int_s^t L_v P_{v,t}  \psi  (x) dv $$ 
such that 
\begin{equation}\label{eq:pstpsi}
P_{s,t} \psi = \psi + \int_s^t  L_v P_{v,t} \psi     dv .
\end{equation}
Plugging this into \eqref{eq:tildeeta} and observing that $ \psi $ and $P_{s, t } \psi , $ and thus, a posteriori, also $\int_s^t  L_v P_{v,t} \psi     dv$ are valid test functions, we obtain 
\begin{equation}\label{eq:notgood}
 \int_0^t < \tilde \eta_s, L_s \psi> ds = \int_0^t < \tilde \eta_s, L_s P_{s, t } \psi > ds - \int_0^t \int_s^t < \tilde \eta_s, L_s  L_u P_{ u,t } \psi > du ds.
\end{equation}
Let us consider the double integral appearing in the above expression. By the definition of $ L_s$  and using equation \eqref{eq:expdecrease} of Proposition \ref{prop:Pst}, we know that 
$$ \Psi_{s,  u, t } := L_s  L_u P_{ u,t } \psi  \in C^4_b, \mbox{ satisfying } \lim_{ x\to \infty} x^\gamma | (\Psi_{s,u, t } \psi )^{(k)} )(x)| = 0 ,$$
for all $ 0 \le k \le 4,$ for any $ \gamma > 0.$ This implies that 
$$ \sup_{ s \le u \le t \le T } \|  L_s  L_u P_{ u,t } \psi  \|_{4,  p } = C_T < \infty $$
such that 
$$ |  < \tilde \eta_s, L_s  L_u P_{ u,t } \psi > | \le \| \tilde \eta_s\|_{- 4,  p } \|  L_s  L_u P_{ u,t } \psi  \|_{4,  p } \le C_T  \| \tilde \eta_s\|_{- 4,  p } . $$ 
Since $ \tilde \eta $ takes values in $C ( \R_+, \W_0^{- 4,  p } ), $   $ \sup_{ s \le t }  \| \tilde \eta_s\|_{- 4,  p }  < \infty, $ and therefore we may use Fubini's theorem and obtain 
\begin{equation}\label{eq:good}
 \int_0^t \int_s^t < \tilde \eta_s, L_s L_u P_{u, t }  \psi > du ds = \int_0^t \int_0^u < \tilde \eta_s ,  L_s L_u P_{u, t } \psi > ds du .
\end{equation}   

Now we apply \eqref{eq:tildeeta} to the admissible test function $ \varphi := L_u P_{u, t } \psi , $ for fixed $ u < t ,$ at time $u.$ Then 
$$ < \tilde \eta_u , \varphi > = \int_0^u < \tilde \eta_s , L_s \varphi > ds + R_u,$$
where 
\begin{equation}\label{eq:ru}
R_u =   H_u ( L_u P_{u, t } \psi )  \mbox{ and where we write for short }  H_u ( \cdot) := \int_0^u h_s ( \cdot ) ds .
\end{equation}
Notice that
$$ \int_0^u < \tilde \eta_s , L_s \varphi > ds = \int_0^u < \tilde \eta_s , L_s L_u  P_{u, t } \psi > ds .$$ 
As a consequence, the double integral in \eqref{eq:good} can be rewritten as 
\begin{multline}\label{eq:alsogood}
  \int_0^t \int_0^u < \tilde \eta_s , L_s L_u P_{u, t }  \psi > ds du = \int_0^t < \tilde \eta_u , L_u P_{u, t } \psi > du - \int_0^t R_u du\\
   =\int_0^t < \tilde \eta_s , L_s P_{s, t } \psi > ds - \int_0^t R_u du .
\end{multline}   
In the last line, we have just changed the integration variable $u$ to the new integration variable $s$ such that the comparison to  the first term in the rhs of \eqref{eq:notgood} is easier. Indeed,  we see that \eqref{eq:notgood} together with \eqref{eq:alsogood} now implies that
$$  \int_0^t < \tilde \eta_s, L_s \psi> ds =  \int_0^t R_u du, \mbox{ where, using \eqref{eq:ru}, }
\int_0^t R_u du = \int_0^t H_u (L_u P_{u, t } \psi)  du .
$$

Using the same trick as above, 
$$ H_t (\psi ) = \int_0^t h_s ( \psi) ds = \int_0^t h_s (P_{s,t} \psi) ds - \int_0^t \int_s^t h_s  (  L_v P_{v,t} \psi ) dv ds  .$$ 
Proposition \ref{prop:hs}, with $ k =1, $ $q=2 \alpha $ together with Lemma \ref{prop:5} implies
$$ | h_s ( L_v P_{v, t } \psi )| \le C_t \| L_v P_{v, t } \psi\|_{1, 2 \alpha} \| \tilde \eta_s\|_{- 4,  p } \| f \|_{4,  p }  \le C_t  \|  P_{v, t } \psi\|_{2,   \alpha} \| \tilde \eta_s\|_{- 4,\bar  p } \| f \|_{4, \bar p },$$
which is bounded uniformly in $ 0 \le s \le v \le t , $ due to Proposition \ref{prop:Pst}, since $ \psi \in C_c^\infty. $  Therefore, we may use Fubini's theorem once more to deduce that 
$$ \int_0^t \int_s^t h_s  (  L_v P_{v,t} \psi ) dv ds  = \int_0^t H_v (L_v P_{v, t } \psi ) dv .$$ 
Gathering all these terms, we end up with 
\begin{equation}\label{eq:almostthere}
< \tilde \eta_t , \psi > = \int_0^t h_s ( P_{s,t} \psi) ds .
\end{equation}

We are now ready to finish this proof. Equality \eqref{eq:almostthere} together with Proposition \ref{prop:hs} applied with $k=4$ and $q=p$ and Proposition \ref{prop:Pst} imply 
that for all $ t \le T$ and all $   \psi \in C_c^\infty, $ 
$$ | < \tilde \eta_t , \psi >| \le C_T \| \psi\|_{4,  p } \|f\|_{4,  p } \int_0^t \| \tilde \eta_s\|_{-4, p }   ds . $$
Since $ C^\infty_c $ is dense in $ \W_0^{4,  p },$ this implies
$ \| \tilde \eta_t\|_{- 4,  p } \le C_T \int_0^t  \| \tilde \eta_s\|_{-4, p } ds , $
and Gronwall's lemma implies $  \| \tilde \eta_t\|_{-4, p } = 0 $ for all $ t \geq 0.$   $\qed.$

\section{Decomposition of the fluctuations}\label{sec:3bis}
We now turn to the second main part of this paper and propose a first decomposition of the fluctuation measure $ \eta^N $ for a fixed system size $N.$ The following purely discontinuous martingale, defined for any measurable bounded test function $ \varphi ,$ will play a key role in our study.   
\begin{equation}\label{eq:WtN}
 W_t^N ( \varphi ) = \frac{1}{\sqrt{N}} \sum_{i=1}^N \int_{[0, t ] \times \R_+} \varphi (X^{N,i }_{s-})  \indiq_{\{ z \le f ( X^{N, i }_{s-} ) \}} \tilde \pi^i (ds, dz) ,
\end{equation}
where $ \tilde \pi^i (ds,dz ) = \pi^i (ds, dz) - ds dz$ is the compensated Poisson random measure. Clearly, $ (W_t^N ( \varphi ))_{t \geq 0 } $ is a real valued martingale with angle bracket given by 
\begin{equation}\label{eq:anglebracketwtn}
 < W^N ( \varphi ) >_t = \int_0^t \mu_s^N ( f \varphi^2 ) ds .
\end{equation} 
We obtain the following first decomposition of $ \eta_t^N ( \varphi ), $ for sufficiently smooth test functions $ \varphi.$

\begin{proposition}\label{prop:first}
Grant Assumptions \ref{ass:1} and \ref{ass:2}. Then  for any test function $ \varphi \in C^2_b $ and $t \geq 0, $  
\begin{multline}\label{eq:decomp}
\eta_t^N ( \varphi) = \eta_0^N ( \varphi ) + \int_0^t \eta_s^N ( L_s \varphi ) ds + W_t^N ( R \varphi )\\
+ h \int_0^t \mu_{s-}^N ( \varphi ') d W_s^N ( 1) + h \int_0^t \eta_s^N ( f) \mu_s^N ( \varphi ' ) ds + R^N_t ( \varphi ) ,
\end{multline}
where the remainder term is given by
\begin{multline}\label{eq:RtN}
 R_t^N ( \varphi ) \\
 =  \frac{h }{ N^{3/2} } \sum_{i=1}^N \int_{[0, t ]\times \r_+}\indiq_{\{ z \le f( X^{N, i }_{s-} )\}} \left( 
\left[   \sum_{j=1, j \neq i }^N 
\varphi ' ( X^{N, j }_{s-} + \theta _s \frac{h}{N} ) - \varphi' ( X^{N, j }_{s-} \right]  - \varphi ' (X^{N, i }_{s-} ) \right)  \pi^i (ds, dz)  ,
\end{multline}
for some $ \theta_s \in [0, 1 ] .$ 
\end{proposition}

\begin{proof}
Using Taylor's formula at order two, we obtain for any $ \varphi \in C_b^2 ,$ 
\begin{eqnarray*}
\mu_t^N (\varphi) &=& \mu^N_0 ( \varphi ) - \alpha \int_0^t \mu_s^N ( \varphi' \cdot x ) ds + \frac{1}{\sqrt{N}} W_t^N ( R \varphi )+ \int_0^t \mu_s^N ( f R \varphi)ds  \\
&&+  \frac{h}{\sqrt{N}} \int_{[0, t ]} \mu_{s-}^N ( \varphi ') d W_s^N ( 1) 
+ \int_0^t \mu_s^N (f) h \mu_s^N ( \varphi ' ) ds + \frac{1}{\sqrt{N}}  R_t^N (\varphi ) \\
&=&   \mu^N_0 ( \varphi )  + \int_0^t \mu_s^N ( L_s \varphi ) ds  + \frac{1}{\sqrt{N}} W_t^N ( R \varphi ) \\
&&+  \frac{h}{\sqrt{N}} \int_0^t \mu_{s-}^N ( \varphi ') d W_s^N ( 1) 
+h \int_0^t ( \mu_s^N ( f) - g_s (f) ) \mu_s^N ( \varphi' ) ds 
+\frac{1}{\sqrt{N}} R_t^N (\varphi ) .
\end{eqnarray*}
 In the above development we have used that 
$$ \frac{1}{N} \sum_{i=1}^N \int_{[0, t ] \times \R_+} \indiq_{\{ z \le f (X^{N, i }_{s-} ) \}} \left( \frac{h}{N} \sum_{j=1}^N \varphi ' ( X^{N, j }_{s-} ) \right) \tilde \pi^i ( ds, dz) = \frac{h}{\sqrt{N}} \int_0^t \mu_{s-}^N ( \varphi ') d W_s^N ( 1) .$$

Using that
$ g_t ( \varphi ) = g_0 ( \varphi ) + \int_0^t g_s ( L_s \varphi ) ds ,$
By Ito's formula, we obtain the result.
\end{proof}

We now give estimates of the terms $ \eta^N , W^N , R^N$ appearing in \eqref{eq:RtN} above, interpreted as elements of $ \W_0^{- k , p}, $ for the smallest possible $k, p .$  This will be useful later to deduce the tightness of these processes. 

\begin{proposition}\label{prop:8}
Grant Assumptions \ref{ass:1} and \ref{ass:2}. Then for any $ p > 1/2 $ and any $ T > 0, $  
$$ \sup_{t \le T} \sup_N \E ( \| \eta_t^N\|_{ - 2, p } ) < \infty .$$ 
\end{proposition}

\begin{remark}\label{rem:2} 
We stress that we obtain a weaker result than the corresponding Proposition 3.5 in \cite{fernandez} or Proposition 4.7 in \cite{chevallier} since we are not able to control the expectation of the square of the norm $\E ( \| \eta_t^N\|^2_{ - 2, p } ).$  This is due to two facts. 

{\bf Fact 1.} We are  working in the framework of point processes, not of diffusions. Therefore, the control 
$$ \E | \bar X^i_t - X_t^{N, i } | \le C_T N^{-1/2} $$
given in \eqref{eq:prop1} cannot be improved to higher order moments of the strong error as in \cite{fernandez}. This intrinsic difficulty is common to any study of point processes.

{\bf Fact 2.} Julien Chevallier in \cite{chevallier} proposes to remediate this difficulty by considering rather higher order moments of the total variation distance; that is, proving and exploiting the fact that 
$$ \P ( \| Z^{N, i } - \bar Z^i \|_{TV, [0, T]} \neq 0 \mbox{ for all } 1 \le i \le k )  \le C_T N^{- k/2}.$$
However, in our model, even on $\{\| Z^{N, i } - \bar Z^i \|_{TV, [0, T]} = 0 \} ,$ the two processes do not couple since they are driven by two different drift terms. This is a crucial difference with the age-structured Hawkes process where the drift is always $ \equiv 1, $ independently of anything else (compare more precisely to (A.10) of \cite{chevallier}). It is for the same reason that we have to take test functions that are twice continuously differentiable, such that we work in $ \W_0^{- 2, p } .$ 
\end{remark}

\begin{corollary}
Since $f \in C^{2, \alpha} \subset \W_0^{2, p } $ for any $ p > \alpha + \frac12, $ such that $ | \eta_t^N ( f) | \le \| \eta_t^N\|_{- 2 ,p } \| f \|_{2, p }, $ we deduce from Proposition \ref{prop:8} the useful upper bound 
\begin{equation}\label{eq:etatNf}
\sup_{t \le T} \sup_N \E ( | \eta_t^N( f) | ) < \infty.
\end{equation}
\end{corollary}

\begin{proof}[Proof of Proposition \ref{prop:8}]
Let $ \bar X^{i, }, 1 \le i \le N, $ be independent copies of the limit system \eqref{eq:dynlimit}, driven by the same Poisson random measures as $ X^{N, i }, 1 \le i  \le N, $ and starting from the same initial positions $ X^i_0 , 1 \le i \le N, $ as the finite system. We decompose, for any $ \psi \in \W_0^{2, p }, $  
$$ \eta_t^N ( \psi) = \sqrt{N} \left( \frac1N \sum_{i=1}^N [\psi ( X^{N, i }_t ) - \psi ( \bar X^i_t) ] + [ \psi ( \bar X^i_t) - \E ( \psi ( \bar X^i_t) ) ] \right) =: \eta_t^{N, 1} ( \psi) + \eta_t^{N,2} ( \psi) ,$$
such that  $ \| \eta_t^N\|_{-2, p } \le \| \eta_t^{N,1} \|_{- 2, p } + \| \eta_t^{N,2} \|_{-2,p }.$  

{\bf Step 1.} We take an orthonormal basis $(\psi_k)_k $ composed of $C^\infty_c-$ functions of $ \W_0^{2, p} $ such that 
$$ \| \eta_t^{N,2} \|^2_{-2, p } = \sum_k {< \eta_t^{N,2} , \psi_k >}^2.$$ 
Using the independence of the $ \bar X^i, i \geq 1, $ we have 
$$
\E (< \eta_t^{N,2} , \psi_k >^2)  =  \E ( \frac1N \sum_{i=1}^N [ \psi_k ( \bar X^i_t) - \E ( \psi_k ( \bar X^i_t) ) ]^2 ) = \E(  [ \psi_k ( \bar X^1_t) - \E ( \psi_k ( \bar X^1_t) ) ]^2) \le \E (\psi_k^2 ( \bar X^1_t) ) .$$ 
Observing that 
$  \E (\psi_k^2 ( \bar X^1_t) ) = 
\E ( < \delta_{\bar X^1_t}, \psi_k>^2) ,$
we obtain by monotone convergence
\begin{multline*}
 \E \| \eta_t^{N,2}\|^2_{-2, p } =   \E \left(\sum_k (< \eta_t^{N,2} , \psi_k >^2)\right)   = \sum_k \E  (< \eta_t^{N,2} , \psi_k >^2)  \le \sum_k \E (\psi_k^2 ( \bar X^1_t) )\\
=  \sum_k  \E ( < \delta_{\bar X^1_t}, \psi_k>^2) = \E \left(\sum_k  < \delta_{\bar X^1_t}, \psi_k>^2 \right)= 
\E \|\delta_{\bar X^1_t} \|^2_{- 2, p } .
\end{multline*}
Thanks to \eqref{eq:deltax} together with \eqref{eq:goodtoknow}, we have that 
$\|\delta_{\bar X^1_t} \|_{- 2, p } \le C ( 1 + |\bar X^1_t|^p ) \le \bar C_T,$
where we have used the a priori estimate \eqref{eq:aprioribound2}. As a consequence,  
$$ \sup_{ t \le T } \sup_N \E \| \eta_t^{N,2} \|_{-2, p }  \le C_T  \| \psi \|_{2, p }  .$$

{\bf Step 2.}
We now study the first term. For any $ \psi \in \W_0^{2, p} ,$ using that for any $x, y \geq 0, $ by Taylor's formula and the Sobolev embedding, 
$$ | \psi ( x) - \psi (y) | \le C \| \psi\|_{C^{1, p }} ( 1 + |x|^p + |y|^p ) |x-y| \le C \| \psi\|_{2, p }  ( 1 + |x|^p + |y|^p ) |x-y|,$$
we obtain, using the upper bound \eqref{eq:aprioribound2}, 
$$
 | \psi ( X^{N, i }_t ) - \psi ( \bar X^i_t)|  \le \bar C_T  \| \psi \|_{2, p }  ( 1 + | X^{N, i }_t|^p)  | X^{N,i }_t - \bar X^i_t| ,
$$
such that 
$$\| \eta_t^{N,1} \|_{-2, p } =  \sup_{ \psi :  \| \psi \|_{2, p } = 1} | \eta_t^{N,1} ( \psi ) |\le \bar C_T \frac{1}{ \sqrt{N} } \sum_{i=1}^N  ( 1 + | X^{N, i }_t|^p  ) | X^{N,i }_t - \bar X^i_t|  .$$

Recall the set $G_T^N $ introduced in \eqref{eq:gtn} above. 
On the set $ G^N_T ,$ using \eqref{eq:aprioriboundgt}, we have that $  \sup_{t \le T } | X^{N, i }_t|^p  \le C_T ,$ whence
$$  \E ( \| \eta_t^{N,1} \|_{-2, p } ; G_T ) \le \frac{C_T}{\sqrt{N} }   \sum_{i=1}^N \E (  | X^{N, i }_t - \bar X^i_t| ) .$$
We then deduce from  \eqref{eq:prop1} that
$$  \E ( \| \eta_t^{N,1} \|_{- 2, p } ; G^N_T ) \le C_T  . $$
 
Moreover, on $(G_T^N)^c,$ we simply upper bound, using once more \eqref{eq:aprioribound2}, 
$$ \| \eta_t^{N,1} \|_{-2, p } \le \bar C_T \frac{1}{ \sqrt{N} } \sum_{i=1}^N  ( 1 + | X^{N, i }_t|^p  ) \left( X^{N,i }_t + \bar C_T  \right)  \le C_T \frac{1}{ \sqrt{N} } \sum_{i=1}^N  ( 1 + | X^{N, i }_t|^{p+1}  ),$$
where we recall that constants may change from one appearance to another and where we have used that $ (1 + x^p) (1+x) \le C (1+x^{p+1}), $ for a suitable constant.
Therefore,  
$$ \E (\| \eta_t^{N,1} \|_{-2, p } ; (G_T^N)^c) \le C_T \sqrt{N} \E ( (1 + | X^{N, i }_t|^{p+1}  ) \indiq_{(G_T^N)^c} ) .$$ 
Using \eqref{eq:aprioribound1} with $2(p+1) $ and the Cauchy-Schwartz inequality together with $ \mathbb{P} ( (G_T^N)^c) \le a e^{ -b N T},$
this gives
$$ \E (\| \eta_t^{N,1} \|_{-2, p }; (G_T^N)^c ) \le C_T \sqrt{N} e^{ - (b/2)  NT}.$$ 
All in all we therefore get 
$$ \sup_N \sup_{ t \le T}  \E\| \eta_t^{N,1} \|_{-2, p }  \le C_T ,$$
and this concludes the proof. 
\end{proof}

\begin{proposition}\label{prop:9}
Under Assumptions \ref{ass:1} and \ref{ass:2}, for any $ p > 1/2 ,$ the process $ W^N_t $ is a $ ({\mathcal F}_t)_{t \geq 0} -$martingale with paths in $ D ( \R_+, \W_0^{- 1, p } ) $ almost surely. Furthermore, 
\begin{equation}\label{eq:411}
 \sup_N \E \left( \sup_{t \le T } \| W^N_t\|^2_{ - 1, p } \right) < \infty .
\end{equation}
\end{proposition}

\begin{proof}
Take an orthonormal basis $ (\psi_k)_{ k \geq 1 }$ of $ \W_0^{1, p}, $ composed of $ C^\infty_c -$functions, and use that 
$$\sup_{ t \le T }  \| W_t^N \|^2_{-1, p } = \sup_{t \le T }  \sum_{ k \geq 1} (W_t^N ( \psi_k))^2 \le \sum_k \sup_{t \le T } (W_t^N ( \psi_k))^2 .$$
As a consequence, by Doob's inequality and monotone convergence, and relying on \eqref{eq:anglebracketwtn},
\begin{multline*}
 \E ( \sup_{ t \le T }  \| W_t^N \|^2_{-1, p } ) \le 4 \sum_k \E (W_T^N ( \psi_k))^2= 4 \sum_k \E \int_0^T \mu_s^N ( f \psi^2_k)ds = 4 \sum_k \E \int_0^T f ( X_s^{N, 1 })  \psi_k^2 ( X_s^{N, 1 } ) ds  \\
 =  4 \E \int_0^T  f ( X_s^{N, 1 }) \sum_k \psi_k^2 ( X_s^{N, 1} ) ds = 4  \E \int_0^T  f ( X_s^{N, 1 }) \| \delta_{X_s^{N, 1 }} \|^2_{ - 1, p } ds ,
\end{multline*} 
where we have used the exchangeability of the finite system to obtain the last term of the first line.

By Lemma \ref{lem:6}, there exists a constant not depending on $ X_s^{N, 1 } $ such that 
$  \| \delta_{X_s^{N, 1 }} \|^2_{ - 1, p }\le C  ( 1 + | X_s^{N, 1} |^{2p } ) .$
Moreover, $  f ( X_s^{N, 1 }) \le C ( 1 + |X_s^{N, 1 }|^\alpha) .$ 
Using \eqref{eq:aprioribound1} with $ 2p +\alpha, $ this implies \eqref{eq:411}. 

Once  \eqref{eq:411} is checked, the remainder of the proof follows the lines of the proof of Proposition 4.7, item (ii) of \cite{chevallier}. 
\end{proof}

We now  check that 

\begin{proposition}\label{prop:10}
Grant Assumptions \ref{ass:1} and \ref{ass:2}. For all $p > 1/2$  we have 
$$ \sup_N \sqrt{N} \E ( \sup_{t \le T }  \| R_t^N\|_{- 3, p } ) < \infty .$$
\end{proposition}

\begin{proof}
Let $ \psi \in \W_0^{3, p } $ and recall that, by Taylor's formula and the Sobolev embedding, 
$$ | \psi' ( x + \theta \frac{h}{N} ) - \psi'( x)| \le  C \| \psi\|_{C^{2, p } } ( 1 + x^p ) \frac{h}{N} \le C \| \psi\|_{3, p } (1+x^p ) \frac{h}{N}.$$
Therefore, 
$$
\sqrt{N}  | R_t^{N } ( \psi ) |  \le  C \| \psi\|_{ 3, p}
 \frac{h }{ N } \sum_{i=1}^N \int_{[0, T ]\times \r_+} \indiq_{\{ z \le f( X^{N, i }_{s-} )\}} \left( 
\frac{h}{N}   \sum_{j=1}^N 
(1 + |X^{N, j }_{s- } |^p) + | X^{N, i }_{s- } |^p  \right) ,
$$
where we have also used that  
$  |\psi' (x) | \le C \| \psi\|_{C^{1, b }} ( 1 + |x|^p ) \le C \| \psi\|_{3, p}( 1 + |x|^p ) .$ This implies 
\begin{multline*}
\sqrt{N}  \sup_{ t \le T } \| R_t^N \|_{ - 3, p } =  \sup_{t \le T } \sup_{ \psi :  \| \psi \|_{3, p } = 1 } \sqrt{N}  | R_t^{N } ( \psi ) |  \le \\
 C \frac{h }{ N } \sum_{i=1}^N \int_{[0, T ]\times \r_+} \indiq_{\{ z \le f( X^{N, i }_{s-} )\}} \left( 
\frac{h}{N}   \sum_{j=1}^N 
(1 + |X^{N, j }_{s- } |^p) + | X^{N, i }_{s- } |^p  \right) .
\end{multline*}
Taking expectation and using the a priori bound \eqref{eq:aprioribound1} together with $ f( x) \le C (1 + x^\alpha) $ yields the result.
\end{proof}

Finally, recall that $ D : \psi \to \psi' $ denotes the differential operator, and $ D^* $ the associated dual. 

\begin{lemma}\label{lemma:12}
Fix  $ p >1/2.$ Then under Assumptions \ref{ass:1} and \ref{ass:2}, the mapping defined by $ \psi \mapsto \eta^N_s ( f) D^* \mu_s^N (\psi) $ is almost surely continuous from $\W_0^{2, p} \to \R $ and satisfies 
$$  \sup_N \sup_{t \le T } \E ( \| \eta_t^N ( f) D^* \mu_t^N  \|_{ - 2, p } ) < \infty .$$
\end{lemma}

\begin{proof}
The result follows from 
$$ |\eta_t^N ( f) D^* \mu_t^N ( \psi) | \le | \eta_t^N ( f) | \frac1N \sum_{i=1}^N | \psi'  ( X_t^{N, i } ) | \le C \| \psi \|_{2,p} | \eta_t^N ( f) | \frac1N \sum_{i=1}^N ( 1 + | X_t^{N, i } |^p ) .$$
The conclusion is then similar as in the proof of Proposition \ref{prop:8}.

\end{proof} 

We now turn to the study of the last stochastic integral appearing in \eqref{eq:decomp}.

\begin{proposition}\label{prop:13}
For any $ p > 1/2 ,$ the process $ \int_0^t D^* \mu_{s-}^N d W_s^N (1) $ is a $ ({\mathcal F}_t)_{t \geq 0} -$martingale with paths in $ D ( \R_+, \W_0^{- 2, p } ) $ almost surely. Furthermore, 
\begin{equation}\label{eq:412}
 \sup_N \E \left( \sup_{t \le T } \| \int_0^t D^* \mu_{s-}^N d W_s^N (1)|^2_{ - 2, p } \right) < \infty .
\end{equation}
\end{proposition}

\begin{proof}
The proof is similar to the proof of Proposition \ref{prop:8}. As there, we take an orthonormal basis $( \psi_k)_{ k \geq 1} ,$ now of $ \W_0^{2, p}, $ composed of $ C^\infty_c -$functions.  We have
$$\sup_{ t \le T }  \| \int_0^t D^* \mu_{s-}^N d W_s^N (1) \|^2_{-2, p } = \sup_{t \le T }  \sum_{ k \geq 1} (\int_0^t \mu_s^N ( \psi'_k) d W_s^N ( 1) )^2 \le  \sum_{ k \geq 1} \sup_{t \le T } (\int_0^t \mu_s^N ( \psi'_k) d W_s^N ( 1) )^2.$$
Applying first  Doob's and then Jensen's inequality and finally monotone convergence, 
\begin{multline*}
 \E ( \sup_{ t \le T } \| \int_0^t D^* \mu_{s-}^N d W_s^N (1) \|^2_{-2, p }) \le 4   \sum_k \E \int_0^T\mu_s^N (f) ( \mu_s^N ( \psi_k'))^2 ds\le  4  \sum_k \E \int_0^T \mu_s^N (f)\mu_s^N ( (\psi_k')^2)) ds \\
 = 4  \E \int_0^T \mu_s^N (f)\left[ \frac1N \sum_{i=1}^N  \sum_k  (\psi_k ')^2 ( X_s^{N, i } )\right]  ds .
\end{multline*} 
Now we rely on Lemma \ref{lem:6} and use that 
$$ \| \delta_x \circ D\|^2_{- 2, p} = \sum_k ( \delta_x \circ D ( \psi_k) )^2 = \sum_k ( \psi_k' ( x) )^2 $$
to identify 
$$ \sum_k  (\psi_k ')^2 ( X_s^{N, i } ) = \| \delta_{X_s^{N, i } } \circ D \|^2_{- 2, p } \le C (1 + | X_s^{N, i } |^{2p} )  $$
such that 
$$  \E ( \sup_{ t \le T } \| \int_0^t D^* \mu_{s-}^N d W_s^N (1) \|^2_{-2, p }) \le C \E \int_0^T \mu_s^N (f) ( 1 + \mu_s^N ( |\cdot|^{2p} ) ) ds ,$$
which, 
together with our a priori estimate \eqref{eq:aprioribound1}, using similar arguments as in the end of the proof of Proposition \ref{prop:9}, allows to conclude that  
$$ \E ( \sup_{ t \le T } \| \int_0^t D^* \mu_{s-}^N d W_s^N (1) \|^2_{-2, p })   \le C_T.$$
The remainder of the assertion follows once more along the lines of  the proof of item (ii) of Proposition 4.7 in \cite{chevallier} 
\end{proof}

To close this section, we state the following
\begin{lemma}\label{lem:17}
Under Assumptions \ref{ass:1} and \ref{ass:2}, for any $ p >  \frac12, $ the integrals 
$$ \int_0^t L_s^* \eta_s^N ds \mbox{   and   }  \int_0^t \eta_s^N ( f) D^* \mu_s^N ds $$
(where $L_s^* $ and $ D^* $ denote the dual operators of $L_s $ and of $D$) 
are almost surely well defined as Bochner integrals in $ \W_0^{ - 3, p }.$ Furthermore,  $ t \mapsto \int_0^t L_s^* \eta_s^N ds $ and $ t \mapsto \int_0^t \eta_s^N ( f) D^* \mu_s^N ds $ are almost surely strongly continuous in $ \W_0^{- 3, p }.$
\end{lemma}
The proof of the above lemma is sketched in the beginning of the proof of Proposition 3.5 in \cite{fernandez}. 

Resuming what we have done so far, we conclude that

\begin{proposition}\label{prop:14}
Grant Assumptions \ref{ass:1} and \ref{ass:2}. Then for any $ p >\alpha + \frac12 ,$ 
we have the decomposition in $\W_0^{ - 3, p } $ 
\begin{equation}\label{eq:12}
\eta_t^N = \eta_0^N + \int_0^t L_s^* \eta_s^N ds + R^* W_t^N + h \int_0^t D^* \mu_{s-}^N  d W_s^N(1) + h \int_0^t  \eta_s^N ( f) D^* \mu_s^N   ds +R_t^N ,
\end{equation}
where $ R^* $ denotes the dual operator of $R : \psi \mapsto \psi ( 0 ) - \psi( \cdot) $ and where $R_t^N$ is given in \eqref{eq:RtN}.

Moreover, 
\begin{equation}\label{eq:uniforminTbound}
 \sup_N \E ( \sup_{ t \le T } \| \eta_t^N\|_{- 3, p } ) < \infty 
\end{equation} 
and $ t \mapsto \eta_t^N$ belongs to $ D( \R_+, \W_0^{-3, p }) $ almost surely. In particular,
\begin{equation}\label{eq:uniforminTboundeta}
\sup_N \E ( \sup_{ t \le T } |\eta_t^N (f) | ) < \infty .
\end{equation}
\end{proposition}

\begin{remark}
The above decomposition is stated in $\W_0^{ - 3, p } $ for any $ p >\alpha + \frac12 .$ This lower bound comes from the fact that we have to apply $ \eta^N_t$ to the jump rate function $f$  which belongs to $ C^{ 6 , \alpha } \subset \W_0^{4, p } $ under the condition $ p > \alpha + \frac12 .$ 
\end{remark}

\begin{proof}
Decomposition \eqref{eq:12} follows from our previous results Proposition \ref{prop:first}--\ref{prop:13}. It implies that 
\begin{multline*}
 \sup_{ t \le T } \| \eta_t^N\|_{-3, p } \le \| \eta_0^N\|_{ - 3, p} + \int_0^T \| L_s^* \eta_s^N\|_{-3, p } ds + | h | \int_0^T \|\eta_s^N ( f)  \D^* \mu_s^N\|_{-3, p } ds + \sup_{ t \le T } \| R^* W_t^N \|_{- 3, p } \\
 + | h | \sup_{ t \le T } \| \int_0^t D^* \mu_{s-}^N d W_s^N (1) \|_{-3, p} + \sup_{ t \le T } \| R_t^N\|_{-3, p }.
\end{multline*}
We know by Lemma \ref{prop:5} that 
$$ \E \int_0^T \| L_s^* \eta_s^N\|_{-3, p } ds \le C_T \sup_{s \le T} \E \| \eta_s^N \|_{- 2, p+\alpha} $$
which is finite by Proposition \ref{prop:8}. Moreover, by Lemma \ref{lemma:12}, 
$$ \E \int_0^T \|\eta_s^N ( f)  \D^* \mu_s^N\|_{-3, p } ds < \infty.$$ 
By continuity of the application $ R, $ the stochastic integral terms have already been treated in Propositions \ref{prop:9} and \ref{prop:13} and the remainder term in Proposition \ref{prop:10} such that the conclusion follows. The proof of the fact that almost surely $ t \mapsto \eta_t^N$ belongs to $ D( \R_+, \W_0^{-3, p }) $ is analogous to the proof of Proposition 4.10 in \cite{chevallier}. Finally we use that $ | \eta_t^N (f) | \le \| \eta_t^N\|_{-3, p} \|f\|_{3, p} $ to deduce \eqref{eq:uniforminTboundeta}. 
\end{proof}

\section{Tightness}\label{sec:4}
This section is devoted to the proof of the tightness of the laws of $ \eta^N $ interpreted as stochastic processes with c\`adl\`ag paths taking values in $ \W_0^{-4 ,p},$ for some $ p > \alpha + \frac12.$ Although the above decomposition \eqref{eq:12} is stated in $\W_0^{-3, p}, $ we shall see in Remark \ref{rem:4} below why we have to add one degree of regularity and consider the process as process taking values in the bigger space $ \W_0^{-4 ,p}.$

As it is classically done, 
we rely on the tightness criterion of Aldous for Hilbert space valued stochastic processes that we quote from \cite{joffe}. This criterion reads as follows. A sequence $ (X^N)_{N \geq 1 }  $  of processes in $ D ( \R_+, \W_0^{-4, p }) ,$ defined on a filtered probability space $ ( \Omega, ({\mathcal F}_t)_{t \geq 0 }, \P ),$ is tight if 

\begin{enumerate}
\item
For every $t \geq 0 $ and every $ \varepsilon > 0 $ there exists a  Hilbert space $ H_0 $ such that the embedding $ H_0 \hookrightarrow \W_0^{-4, p} $ is Hilbert-Schmidt and such that for all $t \geq 0,$ 
\begin{equation}\label{eq:tightness1bis}
\sup_N \E ( \| X_t^N\|_{H_0} ) < \infty .
\end{equation}\item
For all $ \varepsilon_1, \varepsilon_2 > 0 $ and $ T \geq 0 $ there exist $ \delta^* > 0 $ and $ N_0$ such that for all $({\mathcal F}_t)_{t \geq 0 }-$stopping times $ \tau_N \le T , $
\begin{equation}\label{eq:tightness2}
\sup_{ N \geq N_0} \sup_{\delta \le \delta^*} \P ( \| X^N_{ \tau_N + \delta } - X^N_{\tau_N} \|_{ - 4, p } \geq \varepsilon_1 ) \le \varepsilon_2.
\end{equation}
\end{enumerate}

\begin{theorem}\label{theo:tight}
Grant Assumptions \ref{ass:1} and \ref{ass:2}.
Then the sequences of laws of $ \eta^N, $ of  $ W^N $ and of $  \int_0^{\cdot} D^* \mu_{s-}^N  d W_s^N(1)$ are tight in $ D ( \R_+, \W_0^{-4, p }), $ for any $ p > \alpha + \frac12 .$ 
\end{theorem}

\begin{proof}
{\bf Step 1.} We start studying Condition \eqref{eq:tightness1bis}. It is satisfied with $ H_0 = \W_0^{- 2, p+1} $ for $ \eta^N$ as a consequence of Proposition \ref{prop:8} since the embedding $ \W_0^{- 2, p+1} 
 \hookrightarrow \W_0^{-4, p} $ is of Hilbert-Schmidt type, by Maurin's theorem, see Section \ref{sec:sobolev} in the Appendix. For $ W^N$ it even holds with $ H_0 = \W_0^{-1, p+1}, $ by Proposition \ref{prop:9}, and for $ \int_0^{\cdot} D^* \mu_{s-}^N  d W_s^N(1),$ it follows from Proposition \ref{prop:13}, with 
 $H_0= \W_0^{-2, p+1 }.$ 
 
{\bf Step 2.} We now check Condition \eqref{eq:tightness2} for $ W^N.$ By Rebolledo's theorem (see \cite{joffe}, page 40), it is sufficient to show that it holds for the trace of the processes $ <\!\!\!< W^N>\!\!\!> $ where each $ <\!\!\!< W^N>\!\!\!> $ is the linear continuous mapping from $ \W_0^{ 4, p } $ to $ \W_0^{-4, p } $ given for all $ \psi_1, \psi_2 \in \W_0^{4, p}$ by 
$$ < <\!\!\!< W^N>\!\!\!>_t (\psi_1) , \psi_2 > =  \int_0^t \mu_s^N ( f \psi_1 \psi_2) ds .$$
We take an orthonormal basis $ (\psi_k)_k $ of $ \W_0^{4, p}. $ Then
\begin{eqnarray*}
&&| Tr <\!\!\!< W^N>\!\!\!>_{ \tau_N + \delta} - Tr <\!\!\!< W^N>\!\!\!>_{\tau_N} |  \\
&& \quad  \quad = | \sum_k < <\!\!\!< W^N>\!\!\!>_{\tau_N + \delta } ( \psi_k), \psi_k > - <\!\!\!< W^N>\!\!\!>_{\tau_N} ( \psi_k ), \psi_k> | \\
&& \quad  \quad =  \sum_k \int_{ \tau_N}^{\tau_N + \delta} \mu_s^N ( f \psi_k^2 ) ds  =  \int_{ \tau_N}^{\tau_N + \delta}  \sum_k \mu_s^N (f  \psi_k^2 ) ds  \\
&& \quad  \quad =  \int_{ \tau_N}^{\tau_N + \delta} \frac1N \sum_{ i = 1}^N f( X_s^{N, i } ) \sum_k \psi_k^2 ( X_s^{N, i } ) ds =  \int_{ \tau_N}^{\tau_N + \delta} \frac1N \sum_{ i = 1}^N f( X_s^{N, i } ) \| \delta_{X^{N, i }_s} \|^2_{-4, p } ds .
\end{eqnarray*} 
By Lemma \ref{lem:6}, there exists a constant $C$ with $ \| \delta_{X^{N, i }_s} \|^2_{-4, p }  \le C ( 1 +  |X^{N, i }_s|^{2p}) $ such that we may upper bound the above expression by 
$$ C  \delta \;   \frac{1}{N} \sum_{i=1}^N \sup_{ s \le T + \delta}   ( 1 +  |X^{N, i }_s|^{2p}) (1 +|X^{N, i }_s|^\alpha) $$
having expectation which is upper bounded  uniformly in $N$ by 
$ C_T \delta^* $
thanks to our a priori estimates \eqref{eq:aprioribound1}. This implies \eqref{eq:tightness2} for $W^N.$ 

We now turn to the study of Condition \eqref{eq:tightness2} for the martingale $ M^N :=  \int_0^{\cdot} D^* \mu_{s-}^N  d W_s^N(1).$ We have
$$ < <\!\!\!<M^N>\!\!\!>_t (\psi_1) , \psi_2 > =  \int_0^t \mu_s^N ( f ) \mu_s^N (\psi_1 ') \mu_s^N( \psi_2 ') ds $$
such that, using Jensen's inequality, 
 \begin{multline*}
| Tr <\!\!\!< M^N >\!\!\!>_{ \tau_N + \delta} - Tr <\!\!\!< M^N >\!\!\!>_{\tau_N} |  \\
= | \sum_k < <\!\!\!< M^N >\!\!\!>_{\tau_N + \delta } ( \psi_k), \psi_k > - <\!\!\!< M^N >\!\!\!>_{\tau_N} ( \psi_k ), \psi_k> | \\
=  \sum_k \int_{ \tau_N}^{\tau_N + \delta} \mu_s^N ( f ) (\mu_s^N(\psi_k '))^2  ds  \le   \int_{ \tau_N}^{\tau_N + \delta}  \mu_s^N ( f ) \sum_k \mu_s^N ( (\psi_k') ^2 ) ds ,
\end{multline*} 
and the conclusion follows similarly. 

Finally, using decomposition \eqref{eq:12} and the fact that the sequence of laws of $ R^*W^N$ (by continuity of $R$) and of $ M^N$ have already been shown to be tight, to show the tightness of $ \eta^N, $ it suffices to check condition  \eqref{eq:tightness2} for the remaining terms 
$${\cal R}_t^N=  \eta_0^N + \int_0^t L_s^* \eta^N_s ds + h \int_0^t \eta_s^N ( f) D^* \mu_s^N ds + R^N_t.$$
We have 
$$
\| {\cal R}^N_{\tau_N + \delta} - {\cal R}^N_{\tau_N} \|_{-4, p } \le \delta \sup_{ s \le T + \delta^* } \left( \| L_s^* \eta_s^N \|_{-4, p } + |h| \; \| \eta_s^N ( f) D^* \mu_s^N\|_{-4, p } \right)
 + \frac{1}{\sqrt{N}} \sup_{ t \le T + \delta^*} \| \sqrt{N} R_t^N\|_{-4, p}.
$$
The last term above is controlled thanks to Proposition \ref{prop:10} above, choosing $ N \geq N_0 $ for $ N_0$ sufficiently large. Using the same arguments as in the proof of Proposition \ref{prop:8}, we have
\begin{multline*}
\sup_{ s \le T + \delta^*} \| \eta^N_s ( f) D^* \mu_s^N \|_{-4, p } \le C_{T+ \delta*}  \sup_{ t \le T+ \delta^* } | \eta_t^N ( f) | 
\\
+ C_T \sqrt{N} \left(  \frac{1}{N} \sum_{ i=1}^N ( 1 + \sup_{ t \le T+\delta^* } | X_t^{N, i }|^p) \right) \left(  \frac{1}{N} \sum_{ i=1}^N(1 + \sup_{t \le T + \delta^* } |X_t^{N, i }|^\alpha )\right) \indiq_{(G_T^N)^c} .
\end{multline*}

Recalling that by \eqref{eq:uniforminTboundeta},  
$ \sup_N \E \sup_{t \le T + \delta^* }  | \eta_t^N ( f) | < \infty ,$
we deduce that 
$$\sup_N  \E \sup_{ s \le T + \delta^*} \| \eta_s ( f) D^* \mu_s^N \|_{-4, p } \le C_{T+ \delta*}.$$

We conclude the proof recalling that by Lemma \ref{prop:5}, 
\begin{equation}\label{eq:tricky}
  \| L_s^* \eta_s^N\|_{-4, p }  \le C_{T+ \delta^*} \sup_{s \le T + \delta^*} \| \eta_s^N \|_{- 3, p+\alpha} ,
\end{equation}  
which, together with \eqref{eq:uniforminTbound} implies the assertion.
\end{proof}

\begin{remark}\label{rem:4}
The above proof relies on the decomposition \eqref{eq:12} and on the uniform in time upper bound \eqref{eq:uniforminTbound} which have been stated in $ \W_0^{-3, p} .$ However, the presence of  the integral $ \int_0^t L_s^* \eta_s^N ds ,$ canceling one order of derivative  and the fact that \eqref{eq:uniforminTbound} does only hold in $ \W_0^{-3, p } $ imply that we have to work in $ \W^{- 4, p }_0$ to be able to obtain the tightness of all terms. 
This is a crucial difference with respect to \cite{fernandez} and  \cite{chevallier}. They both use the upper bound 
$$ \E \int_{\tau_N}^{\tau_N+ \delta}  \| L_s^* \eta^N_s\|_{ - k ,  p }  ds \le \delta \E \int_{\tau_N}^{\tau_N +\delta}  \| L_s^* \eta^N_s\|^2_{ - k ,  p }  ds$$
and are hence able to use a non-uniform in time bound on the expectation of the square of the operator norm of $ \eta^N_t.$ Since we are not able to control the square of the operator norm within our framework, see  Remark \ref{rem:2} above, the prize to pay is to impose one degree of regularity more, as we did here.
\end{remark}

\begin{proposition}
Under Assumptions \ref{ass:1} and \ref{ass:2}, for any $ p > 1/2,$ the limit laws of $ \eta^N,$ of $ W^N$ and of $ \int_0^\cdot D^*  \mu^N_{s-} d W^N_s ( 1) $ are supported in $ C ( \R_+, \W_0^{-4, p }).$ 
\end{proposition}

\begin{proof}
Following \cite{bill}, Theorem 13.4, it suffices to show that the maximal jump size within a fixed time interval converges to $0$ almost surely.   
Let us check this for $W_t^N ,$ for $ t \in [0, T ].$ We have for any $ \psi \in \W_0^{4, p}, $
\begin{equation}\label{eq:ztni}
 \Delta W_t^N ( \psi) =\frac{1}{\sqrt{N}} \sum_{i=1}^N \psi ( X_{t-}^{N, i } ) \Delta Z_t^{N, i },  \mbox{ where } 
 Z_t^{N, i } = \int_{ [0, t ] \times \r_+} \indiq_{\{ z \le f ( X_{s-}^{N, i } ) \}} \pi^i ( ds, dz ) .
\end{equation} 
As a consequence,  for all $ \psi \in \W_0^{4, p } $ with $ \| \psi\|_{4, p} = 1, $ since at each jump time $t,$ only one of the processes $ Z^{N, i } $ has a jump, 
$$\sup_{ t \le T }  |  \Delta W_t^N ( \psi)  | \le (C/ \sqrt{N}) (1 + \sup_{1 \le i \le N}  \sup_{t \le T } |X^{N, i }_t|^p) .$$ 
Thanks to Assumption \ref{ass:2}, using \eqref{eq:poisson},
$ \sup_{1 \le i \le N} \sup_{t \le T } |X^{N, i }_t|^p \le C( 1  + |N_T^N |^p) . $
By the strong law of large numbers, almost surely, 
$$ \lim_{N \to \infty } |N_T^N |^p = T^p f(2h)^p < \infty  $$
such that almost surely,
$$  \lim_{N \to \infty } \sup_{ t \le T }  |  \Delta W_t^N ( \psi)  | = 0.$$
Similarly, since $ \Delta \eta_t^N = \sqrt{N} \Delta \mu_t^N, $ 
$$ \Delta \eta_t^N (\psi) = \frac{1}{\sqrt{N}} \sum_{i=1}^N  \left( R \psi ( X_{t-}^{N, i } ) + \frac{h}{N} \sum_{j \neq i} \psi ' ( X^{N, j }_{t-} + \theta_t \frac{h}{N}) \right) \Delta Z_t^{N, i } ,$$
for some $ \theta_t \in (0, 1 ) , $ and
$$ \Delta  \int_0^t \mu^N_{s-} ( \psi ') d W^N_s ( 1)= \frac{1}{\sqrt{N}} \sum_{i=1}^N \left( \frac{1}{N} \sum_j \psi' ( X_{t-}^{N, j } )\right) \Delta Z_t^{N, i } ,$$
and these terms are treated analogously. 
\end{proof}

We close this section with the following
\begin{theorem}\label{theo:tightU}
Grant Assumptions \ref{ass:1} and \ref{ass:2}.
Then the sequence of laws of $( U^{N, i })_{1\le i \le n } $ is tight in $ D ( \R_+,\R^n ). $ 
\end{theorem}

\begin{proof}
Our proof relies once more on the criterion of Aldous, now stated for $\R^n-$valued processes having c\`adl\`ag paths, see Jacod and Shiryaev \cite[Theorem VI. 4.5 page 356]{js}. More precisely, writing $ U^N = (U^{N, 1} , \ldots, U^{N, n})$ and $ \| u \| = \sum_{i=1}^n |u^i |  $ for the $L^1-$ norm on $ \R^n,$   we shall prove that

(a) for all $ T> 0$, all $\varepsilon>0$,
$ \lim_{ \delta \downarrow 0} \limsup_{N \to \infty } \sup_{ (S,S') \in A_{\delta,T}} 
\P ( \|U_{S'}^{N } - U_S^{N  } \| > \varepsilon ) = 0$,
where $A_{\delta,T}$ is the set of all pairs of stopping times $(S,S')$ such that
$0\leq S \leq S'\leq S+\delta\leq T$ a.s.,

(b) for all $ T> 0$, $\lim_{ K \uparrow \infty } \sup_N 
\P ( \sup_{ t \in [0, T ] } \|U_t^{N }\| \geq K ) = 0$.

To show (b), we start with the decomposition 
\begin{multline}\label{eq:utn}
U^{N, i }_t =  - \alpha \int_0^t U^{N,i}_s ds + h W_t^N ( 1) + h \int_0^t \eta_s^N ( f) ds \\
- \int_{ [0, t ] \times \R_+} U^{N, i }_{s- } \indiq_{\{ z \le f ( \bar X^i_{s-} ) \}} \pi^i ( ds, dz ) + \bar R_t^{N,i  } ,
\end{multline}
for all $ 1 \le i \le n,$ where 
\begin{multline} \bar R_t^{N, i } = - \frac{h}{\sqrt{N}}  \int_{ [0, t ] \times \R_+}\indiq_{\{ z \le f ( \bar X^i_{s-} ) \}} \tilde \pi^i ( ds, dz ) \\\
- \sqrt{N}  \int_{ [0, t ] \times \R_+} X^{N, i}_{s- } \left( \indiq_{\{ z \le f ( X^{N, i}_{s-} )  \}} - \indiq_{\{ z \le f( \bar X^i_{s-} ) \}} \right)  \pi^i ( ds, dz )  .
\end{multline}

We first show that \eqref{eq:utn} implies
\begin{equation}\label{eq:uniformupperboundun}
 \sup_N \E ( \sup_{ s \le T } \| U_s^N \| ) < \infty ;
\end{equation} 
once \eqref{eq:uniformupperboundun} is shown, (b) follows immediately. 

To prove \eqref{eq:uniformupperboundun}, notice that \eqref{eq:utn} implies
\begin{multline*}
 \sup_{ s \le T } \| U_s^N \| \le \alpha \int_0^T \| U_s^N\| ds + h\;  n  \sup_{s \le T } |W_s^N( 1) | + h \; n \int_0^T | \eta_s^N (f) | ds \\
+ \sum_{i=1}^n \int_{[0,T] \times \R_+} | U^{N, i }_{s- }| \indiq_{\{ z \le f ( \bar X^i_{s-} ) \}} \pi^i ( ds, dz ) + \sup_{ s \le T } \sum_{i=1}^n | \bar R_s^{N, i } | .
\end{multline*}
By \eqref{eq:prop1}, 
\begin{equation}\label{eq:upperboundut}
\sup_N \E \int_0^T \| U_s^N\| ds \le n T C_T  .
\end{equation}

Moreover, by Burkholder-Davis-Gundy's inequality for discontinuous martingales and \eqref{eq:aprioribound1} once more, 
$$ \E \sup_{ s \le T } | W_s^N(1)|^2 \le C \frac1N \E \int_0^T \bar f ( X^N_s) ds \le C_T, \mbox{ such that }
 \sup_N \E \sup_{ s \le T } | W_s^N(1)| \le \sqrt{C_T}.$$ 
We also use the upper bound 
$ |\eta_s( f) | \le \| \eta_s^N\|_{-3, \alpha+1 } \| f\|_{3, \alpha+1  } $ 
and \eqref{eq:uniforminTboundeta}  to deal with the term  $h \; n  \int_0^T | \eta_s^N (f) | ds .$ 
Moreover, since for all $ s \le T $ and for all $ 1 \le i \le n , $  $ \bar X^i_{s-}  \le \bar C_T $ by \eqref{eq:aprioribound2}  such that $ f( \bar X^i_{s-} ) \le  f( \bar C_T) ,$ 
$$ \sum_{i=1}^n \E  \int_{[0,T] \times \R_+} | U^{N, i }_{s- }| \indiq_{\{ z \le f ( \bar X^i_{s-} ) \}} \pi^i ( ds, dz )\le  f( \bar C_T) \E \int_0^T  | U^N_s| ds $$
which is treated using \eqref{eq:upperboundut}.

Finally to deal with  $ \sup_{t \le T } | \bar R_t^{N, i } |, $ we first observe that the first term appearing in the decomposition of $ \bar R^{N, i } $ satisfies, using once more that $ f ( \bar X^i_{s-} ) \le f ( \bar C_T), $ for all $ s \le T,$
$$  \E | \frac{h}{\sqrt{N}}  \int_{ [0, T ] \times \R_+}\indiq_{\{ z \le f ( \bar X^i_{s-} ) \}} \tilde \pi^i ( ds, dz )| \le T  f( \bar C_T)  2h / \sqrt{N} .$$
Moreover, using the set $ G_T^N $ introduced in \eqref{eq:gtn} above and the fact that $\sup_{ s \le T } |X_s^{N, i } |  \le C_T $ on $G_T^N, $ we have the upper bound for the second  term appearing in the decomposition of $ \bar R^{N, i } $
\begin{multline*}
  \sup_{ t \le T } |\sqrt{N}  \int_{ [0, T ] \times \R_+} X^{N, i}_{s- } \left( \indiq_{\{ z \le f ( X^{N, i}_{s-} )  \}} - \indiq_{\{ z \le f( \bar X^i_{s-} ) \}} \right)  \pi^i ( ds, dz ) | \\
   \le   C_T \sqrt{N}\int_{ [0, T ] \times \R_+} \left| \indiq_{\{ z \le f ( X^{N, i}_{s-} )  \}} - \indiq_{\{ z \le f( \bar X^i_{s-} ) \}} \right|  \pi^i ( ds, dz ) +\\
   \indiq_{ (G_T^N)^c} \sup_{ s \le T } |X_s^{N, i } |  \sqrt{N} \int_{ [0, T ] \times \R_+} [ \indiq_{\{ z \le f ( X^{N, i}_{s-} )  \}} +  \indiq_{\{ z \le f( \bar C_T)  \}} ]    \pi^i ( ds, dz ).
\end{multline*}
The first line of the rhs is treated using \eqref{eq:tv}, the second using the a priori estimates \eqref{eq:aprioribound1} and the deviation estimate on $ \P (G_T^N)^c.$ 
All in all this implies 
$$ \sum_{i=1}^n \sup_N \E   \sup_{ t \le T } |\sqrt{N}  \int_{ [0, t ] \times \R_+} X^{N, i}_{s- } \left( \indiq_{\{ z \le f ( X^{N, i}_{s-} )  \}} - \indiq_{\{ z \le f( \bar X^i_{s-} ) \}} \right)  \pi^i ( ds, dz ) | \le C_T < \infty ,$$ 
and we have just finished the proof of \eqref{eq:uniformupperboundun}.

We finish this step with the observation that $  \sup_{t \le T } | \bar R_t^{N, i } | $ converges to $0 $ in probability, as $ N \to \infty,$ for any $ 1 \le i \le n .$ We only need to consider
\begin{multline*}
  \sup_{ t \le T } |\sqrt{N}  \int_{ [0, t ] \times \R_+} X^{N, i}_{s- } \left( \indiq_{\{ z \le f ( X^{N, i}_{s-} )  \}} - \indiq_{\{ z \le f( \bar X^i_{s-} ) \}} \right)  \pi^i ( ds, dz ) | \\
   \le \sqrt{N} \sup_{t\le T} | X^{N, i}_{t} | \;  \| Z^{N, i } - \bar Z^i\|_{TV, [0, T ] },
\end{multline*}
such that for any $ \varepsilon > 0, $ 
\begin{multline}\label{eq:lastpoint}
 \P (  \sup_{ t \le T } |\sqrt{N}  \int_{ [0, t ] \times \R_+} X^{N, i}_{s- } \left( \indiq_{\{ z \le f ( X^{N, i}_{s-} )  \}} - \indiq_{\{ z \le f( \bar X^i_{s-} ) \}} \right)  \pi^i ( ds, dz ) |  \geq  \varepsilon ) \\
 \le \P (  \| Z^{N, i } - \bar Z^i\|_{TV, [0, T ] } \geq 1 ) \le \E  \| Z^{N, i } - \bar Z^i\|_{TV, [0, T ] } \le C_T N^{ - 1/2 } \to 0  
\end{multline} 
as $N \to \infty ,$ where we have used once more \eqref{eq:tv}.

Finally, (a) follows from the fact that \eqref{eq:utn} implies 
\begin{multline*}
\| U^N_{S'} - U^N_S \| \le 
C  \delta \left(  \sup_{ s \le T } \|U^N_s\| + n \sup_{ s \le T } |\eta_s ( f ) | \right)  \\
+  \sum_{i=1}^n \int_{ [S, S'] \times \r_+}  |U^{N, i }_{s-}|  \indiq_{\{ z \le f ( \bar X^i_{s-}) \}} \pi^i (ds, dz ) + \sup_{t \le T} \sum_{i=1}^n |\bar R_t^{N, i } | \\
+ h n  | W^N_{S'} (1) - W^N_S ( 1) |.
\end{multline*}
The first line of the rhs is treated using \eqref{eq:uniformupperboundun} and \eqref{eq:uniforminTboundeta}. To deal with the second line we use that 
$$\sum_{i=1}^n  \E  \int_{[S,S'] \times \R_+} | U^{N, i }_{s- }| \indiq_{\{ z \le f ( \bar X^i_{s-} ) \}} \pi^i ( ds, dz )\le  f( \bar C_T) \E \int_S^{S'}   \| U^N_s\| ds \le  f( \bar C_T) \delta \E \sup_{s \le T } \|U_s^N\| $$
and that $\sup_{t \le T} \sum_{i=1}^n  |\bar R_t^{N, i } | $ converges to $0$ in probability. Finally, 
$$ \E  | W^N_{S'} (1) - W^N_S ( 1) |^2 = \E \int_S^{S'} \bar f (X^N_s) ds \le \delta \E \sup_{ s \le T } \bar f ( X^N_s) = C_T \delta,$$
which concludes the proof. 
\end{proof}

\section{Characterization of the limit}\label{sec:5}
In this section we study the possible limits of the sequence of $ \eta^N.$ 
Recall the definition of $ W$ in \eqref{eq:wt}. We start with the following preliminary result. 

\begin{proposition}
Grant Assumptions \ref{ass:1} and \ref{ass:2}. Then for any $ p > \alpha + \frac12 , $ the sequence of processes $ W^N$ converges in $ D ( \R_+, \W_0^{-4, p} ) $ to $W.$ 
\end{proposition}

\begin{proof}
We already have proven the tightness of $ W^N.$ To identify any possible limit, consider, for any $ \psi_1, \psi_2 \in \W_0^{4, p } , $ the difference 
$$ < \; <\!\!\!< W^N >\!\!\!>_t (\psi_1) , \psi_2 > - \int_0^t < g_s , \psi_1 \psi_2 f > ds =\int_0^t <\mu_s^N - g_s , \psi_1 \psi_2 f > ds  .$$
We have that 
$$ \E |\int_0^t <\mu_s^N - g_s , \psi_1 \psi_2 f > ds| = \frac{1}{\sqrt{N}}\E |\int_0^t <\eta^N_s  , \psi_1 \psi_2 f > ds|  \to 0 $$
as $N \to \infty ,$ where this last convergence follows from Proposition \ref{prop:14} and Sobolev's embedding theorem. More precisely, since $f \in C^{3, \alpha } $ and $ \psi_1 \in  \W_0^{4, p } \subset C^{3, p } ,$ we have $ \psi_1 \psi_2 f \in \W_0^{3, 2p + \alpha  } $ such that 
$$ |<\eta^N_s  , \psi_1 \psi_2 f > | \le \sup_{s \le t} \| \eta^N_s\|_{3, 2p + \alpha } \| \psi_1 \psi_2 f \|_{3, 2p + \alpha}.$$

Moreover we have already shown that the maximal jump size of $ W^N $ converges to $0$ almost surely. Then the result follows from Rebolledo's central limit theorem for local martingales, following the lines of the proof of Prop. 5.3 in \cite{chevallier}.
\end{proof}

Coming back to the decomposition of $ \eta^N $ in \eqref{eq:12}, we see that we need to consider the joint convergence of $ R^* W^N $ and $ \int_0^\cdot \mu^N_{s-} d W^N_s (1) $ since both 
martingales depend on the same underlying Poisson noise. 

\begin{proposition}Grant Assumptions \ref{ass:1} and \ref{ass:2} and fix $ p >  \alpha + \frac12 . $   Then we have convergence in law in $D ( \R_+, \W_0^{- 4, p } \times \W_0^{4, p } ) $ of $(R^* W^N ,  \int_0^\cdot D^* \mu^N_{s-} d W^N_s (1)) $ to the limit process
$$ ( R^* W , \int_0^\cdot D^* g_s  d W_s ( 1) ).$$
\end{proposition}

\begin{proof}
We have already shown the tightness of $ ( R^* W^N , \int_0^\cdot D^* \mu^N_{s-}  d W^N_s ( 1) ),$ and we know that we have convergence in law $ W^N \to W .$  To prove the above convergence we first decompose 
$$   \int_0^\cdot D^* \mu^N_{s-} d W^N_s (1) = \int_0^\cdot D^* g_s d W_s^N ( 1) + E^N , $$
where 
\begin{equation}\label{eq:en}
 E^N= \int_0^\cdot D^* ( \mu_{s-}^N - g_s) d W_s^N(1) .
\end{equation} 

{\bf Step 1.} We show that $\E [ \sup_{t \le T } \| E^N_t\|^2_{-4, p } ]  \to 0 $ as $ N \to \infty.$ For that sake, let $ ( \psi_k)_k$ be an orthonormal basis of $ \W_0^{4, p}, $ composed of $ C^\infty_c-$functions. We have that 
$$ \sup_{t \le T } \| E^N_t\|^2_{-4, p } \le  \sum_k \sup_{t \le T} \left( \int_0^t  ( \mu_{s-}^N - g_s) ( \psi_k ' )  d W_s^N (1)  \right)^2,$$
such that 
$$ \E  \sup_{t \le T } \| E^N_t\|^2_{-4, p } \le 4  \sum_k \E \int_0^T \mu_s^N(f) [( \mu_{s}^N - g_s) ( \psi_k ' ) ]^2 ds \le  C \E \int_0^T \mu_s^N(f) \| D^* ( \mu_s^N - g_s) \|^2 _{-4, p } ds.$$
On $G_T^N, $ we upper bound  
$$ \| D^* ( \mu_s^N - g_s) \|^2 _{-4, p }  \le C \|  \mu_s^N - g_s \|^2 _{-3, p } \le C \|\mu_s^N - g_s \|_{-3, p }  \sup_{ s \le T } \left(\| \mu_s^N \|_{-3, p} +\|g_s\|_{-3, p } \right)$$
and use that on $G_T^N, $ $  \sup_{ s \le T } \mu_s^N(f) \left( \| \mu_s^N \|_{-3, p} +\|g_s\|_{-3, p }  \right) \le  C_T $ such that 
$$ \E  \left( \sup_{t \le T } \| E^N_t\|^2_{-4, p } ; G_T^N \right) \le C_T \E \int_0^T \|\mu_s^N - g_s \|_{-3, p }ds= \frac{C_T}{\sqrt{N}} \E \int_0^T \|\eta_s^N\|_{-3, p }ds,$$
which converges to $0$ as $N \to \infty $ thanks to Proposition \ref{prop:8}. Moreover, on $ (G_T^N)^c, $ we upper bound 
$$\| D^* ( \mu_s^N - g_s) \|^2 _{-4, p }  \le C \|  \mu_s^N - g_s \|^2 _{-3, p }  \le C_T(1+  \| \mu_s^N\|_{-3, p }^2) .$$
Using \eqref{eq:poisson},
$$ \| \mu_s^N\|^2_{- 3, p } \le C \left( 1 + \left( N_T^N \right)^{2p} \right) $$
and a similar bound for $ \mu_s^N ( f) , $ we obtain 
$$ \E  \left( \sup_{t \le T } \| E^N_t\|^2_{-4, p } ; (G_T^N)^c \right) \le C_T e^{ - (b/2) N T } \to 0 $$
as $ N \to \infty .$ 

As a consequence of this step, it suffices to prove the convergence in law 
$$ ( R^* W^N , \int_0^\cdot D^* g_{s}  d W^N_s ( 1) )\to ( R^* W , \int_0^\cdot D^* g_s  d W_s ( 1) ),$$
as $N \to \infty.$ 

{\bf Step 2.} We now replace the process $ D^* g_s $ serving as integrand by a process which is piecewise constant over time intervals of step-size $ \varepsilon > 0 .$ We put 
$$ \eg_s := g_{\delta (s) } , \;  \delta ( s) = k \varepsilon \mbox{ for all }  k \varepsilon \le s < (k+1) \varepsilon, k \geq 0,$$ 
and let
$$ M^{\varepsilon } := \int_0^\cdot D^* ( \eg_s- g_s) d W_s (1) .$$ 
Using similar arguments as in Step 1, we have 
$$
\E \sup_{s \le T } \| M^{\varepsilon }_s\|^2_{-4, p } \le C \sum_k  \int_0^T g_s ( f) [ (\eg_s - g_s) ( \psi_k ' )]^2 ds \le C_T \int_0^T \| \eg_s - g_s\|^2_{- 3, p } ds .$$
Using Lemma \ref{prop:lemma20} stated in the Appendix below, this last expression is upper bounded by $ C_T \varepsilon^2.$ 

We introduce similarly 
$$ M^{N,  \varepsilon } := \int_0^\cdot D^* ( \eg_s- g_s) d W^N_s (1) $$
and have, since $\sup_N \E \sup_{s \le T } \mu_s^N ( f) \le C_T, $   
$$
\E \sup_{s \le T } \| M^{N, \varepsilon }_s\|^2_{-4, p } \le C \sum_k \E  \int_0^T \mu^N_s ( f) [ (\eg_s - g_s) ( \psi_k ' )]^2 ds \le C_T \int_0^T \| \eg_s - g_s\|^2_{- 3, p } ds  \le C_T \varepsilon^2,$$
where the constant $C_T$ does not depend on $N.$
As a consequence of this step, it suffices to prove the joint convergence of 
$$ ( R^* W^N , \int_0^\cdot D^* \eg_{s}  d W^N_s ( 1) ) \to  R^* W , \int_0^\cdot D^* \eg_{s}  d W_s ( 1) ),$$ 
as $N \to \infty, $ for each fixed $ \varepsilon > 0.$ 

{\bf Step 3.} 
To do so,  it suffices to prove convergence of the marginal laws 
\begin{equation}\label{eq:fidi1}
 (( R^* W^N_{t_1}, \int_0^{t_1} D^* \eg_s d W_s^N (1)  ), \ldots, ( R^* W^N_{t_k}, \int_0^{t_k} D^* \eg_s d W_s^N (1)  ))
\end{equation} 
to the associated limit 
\begin{equation}\label{eq:fidi2}
 (( R^* W_{t_1}, \int_0^{t_1} D^* \eg_s d W_s (1) ), \ldots, ( R^* W_{t_k}, \int_0^{t_k} D^* \eg_s d W_s (1)  )),
\end{equation} 
for any $ k \geq 1, t_1 \le t_2 \le \ldots \le t_k \le T.$  Note that we can rewrite 
$$ \int_0^t D^* \eg_s d W^N_s ( 1) = \sum_{ k : k \varepsilon \le t } D^* \eg_{ k \varepsilon} \left(W^N_{ (k+1) \varepsilon \wedge t } (1)  - W^N_{ k \varepsilon}(1)  \right) .$$
Since  we have convergence in law in $ D ( \r_+, \R)$ of $ W^N ( 1) $ to the limit process $ W(1) $ which is continuous, the above expression converges in law to 
$$  \sum_{ k : k \varepsilon \le t } D^* \eg_{ k \varepsilon} \left(W_{ (k+1) \varepsilon \wedge t } (1)  - W_{ k \varepsilon}(1)  \right)=  \int_0^t D^* \eg_s d W_s ( 1),$$
such that the convergence in law of \eqref{eq:fidi1} to \eqref{eq:fidi2} is indeed implied. Finally,  letting $ \varepsilon \to 0, $ the convergence of the finite dimensional distributions 
$$ (( R^* W^N_{t_1}, \int_0^{t_1} D^* g_s d W_s^N (1)  ), \ldots, ( R^* W^N_{t_k}, \int_0^{t_k} D^* g_s d W_s^N (1)  ))$$
to the associated limit 
$$ (( R^* W_{t_1}, \int_0^{t_1} D^* g_s d W_s (1) ), \ldots, ( R^* W_{t_k}, \int_0^{t_k} D^* g_s d W_s (1)  ))$$
follows, and this concludes the proof.
\end{proof}

We close this section with the following partial result. 

\begin{theorem}\label{theo:char}
Under Assumptions \ref{ass:1} and \ref{ass:2} and for any $ p > \alpha + \frac12 , $  any limit $ ( \bar U, \bar \eta ) $ of $ (U^N, \eta^N)$ is solution in $ D ( \R_+, \R^n ) \times C ( \R_+, \W_0^{- 4, p } ) $ of the following system of stochastic differential equations  
\begin{equation}\label{eq:baruoncemore}
 \bar U^i_t = - \alpha  \int_0^t \bar U^i_s ds + h \int_0^t \bar \eta_s (f) ds   -\int_{[0, t] \times \r_+} \bar U^i_{s-} \indiq_{\{ z \le f ( \bar X^{i }_{s-} ) \}} \pi^i ( ds, dz) + h W_t ( 1) , t \geq 0,  \; 1 \le i \le n, 
\end{equation}
and for any $ \varphi  \in \W_0^{5, p }, $ 
\begin{equation}\label{eq:baretabis}
\bar  \eta_t ( \varphi) = \bar \eta_0 (\varphi) + \int_0^t \bar \eta_s ( L_s \varphi ) ds +h \int_0^t g_s ( \varphi ') \bar \eta_s ( f) ds + W_t( R \varphi )+h  \int_0^t  g_s ( \varphi ')  d W_s ( 1) , t \geq 0. 
\end{equation}
\end{theorem}

\begin{remark}
We have stated the decomposition \eqref{eq:12} in $ \W_0^{-4, p }.$ However, the operator $ L_s$ appearing in \eqref{eq:baretabis} above reduces regularity by one, such that test functions $ \varphi  \in \W_0^{4, p } $ are reduced to test functions in $ \W_0^{3, p+ \alpha}.$ Yet, we have proven tightness of $ (\eta^N)_N$ only in $ \W_0^{-4,  p}, $ such that we have to state the above decomposition in $\W_0^{-5, p },$ although the limit process $ \bar \eta $ takes values in the smaller space $ \W_0^{4, p }.$ This is analogous to Remark  5.7 in \cite{chevallier}. 
\end{remark}

\begin{proof}
{\bf Step 1.} We have already proven the tightness of the sequence of laws of $ U^N.$  
We now consider the c\`adl\`ag process 
$${\mathcal Y}^N :=  (U^{N,i}, \int_{[0, \cdot ] \times \R_+}  U^{N,i }_{s- }  \indiq_{\{ z \le f ( \bar X^i_{s-} ) \}} \pi^i ( ds, dz ), \bar X^i )_{1 \le i \le n }  $$ 
which belongs to $  D ( \R_+, \R^{3n}).$ 

Using analogous arguments as those in the proof of Theorem \ref{theo:tightU} allows to deduce the tightness of the sequence of processes $ {\mathcal Y}^N $ in $ D( \R_+, \R^{3n}) . $ Details of the proof are omitted.

{\bf Step 2.} Due to the previous step and the continuity of any limit law of $ (\eta^N,  W^N, \int_0^{\cdot } D^* g_s d W^N_s(1) ),$  we know that
$$   \left( \left(U^{N,i}, \int_{[0, \cdot ] \times \R_+}  U^{N,i }_{s- }  \indiq_{\{ z \le f ( \bar X^i_{s-} ) \}} \pi^i ( ds, dz ), \bar X^i \right)_{1 \le i \le n } , \eta^N, W^N, \int_0^{\cdot } D^* g_s d W^N_s(1) \right) $$
is tight in 
$$ D ( \R_+, \R^{3n}  \times \W_0^{- 4, p }\times \W_0^{-4, p } \times \W_0^{-4, p }).$$ In what follows we assume without loss of generality that the above sequence converges to some limit 
$$ \left( (\bar U^i, \bar V^i, \bar X^i)_{1 \le i \le n}, \bar \eta, W, \int_0^\cdot D^* g_s d W_s (1) \right),$$
where, to simplify notation, we use the same letter $ \bar X^i$ to denote the limit process as well as the one defining the third coordinates of $ {\mathcal Y}^N.$

To identify the limit, let $(\psi_k)_k$ be an orthonormal basis of $ \W_0^{4, p },$ composed of $ C^\infty_c-$functions. Define for any $k$ the functional $F_k : D ( \R_+,  \W_0^{- 4, p }\times \W_0^{-4, p } \times \W_0^{-4, p }) \to D ( \R_+, \R ) $ by 
\begin{multline*}
 F_k ( f^1, f^2, f^3 )_t = <f^1_t, \psi_k> - <f^1_0, \psi_k> - \int_0^t <f^1_s, L_s \psi_k> ds -  h \int_0^t <f_s^1, f> < g_s, \psi_k'> ds\\
  -  <f^2_t, R \psi_k > -  h <f_t^3, \psi_k > 
\end{multline*}
and 
$ G : D ( \R_+, \W_0^{- 4, p } \times \W_0^{ - 4 , p }  \times \R ) \to D ( \R_+ , \R ) $ by 
$$G ( f^1, f^2, g^1)_t = g^1_t - g^1_0 +\alpha \int_0^t g^1_s ds - h \int_0^t <f^1_s ,f> ds  - h <f^2_t ,1> .$$
Then the system \eqref{eq:baruoncemore}--\eqref{eq:baretabis} is equivalent to 
$$ \mbox{For all } k \geq 1, \; F_k ( \bar \eta, W, h \int_0^\cdot D^* g_s d W_s ( 1) ) = 0 , \; G ( \bar U^i, \bar \eta, W) =  -\int_{[0, \cdot] \times \r_+} \bar U^i_{s-} \indiq_{\{ z \le f ( \bar X^{i }_{s-} ) \}} \pi^i ( ds, dz)  ,$$
for all $ 1 \le i \le n.$

{\bf Step 2.1.} In this step we show that $F_k ( \bar \eta, W, h \int_0^\cdot D^* g_s d W_s ( 1) ) = 0.$ We first prove that $ F_k$ is continuous at every point 
$ ( f^1, f^2, f^3 ) \in C ( \R_+,  \W_0^{- 4, p }\times \W_0^{-4, p } \times \W_0^{-4, p }) .$ Indeed, the continuity of 
$$ f^1 \mapsto \left( t \mapsto <f^1_t, \psi_k> - <f^1_0, \psi_k> - \int_0^t <f^1_s, L_s \psi_k> ds -   h \int_0^t <f_s^1 ,f> < g_s, \psi_k'> ds\right) $$
at every point $f^1 \in C ( \R_+, \W_0^{- 4, p } )$ follows as in the proof of Theorem 5.6 of \cite{chevallier}.
Similarly, the functionals
$$ f^2 \mapsto \left( t \mapsto  <f^2_t, R \psi_k >\right) \mbox{ and } f^3 \mapsto \left( t \mapsto < f^3_t, \psi_k > \right) $$
are continuous as well at every point $ f^2$ and $ f^3  $ belonging to $  C ( \R_+, \W_0^{- 4, p } ) .$ 

{\bf Step 2.2.} Before proceeding further, we rewrite \eqref{eq:12} as follows
\begin{equation}\label{eq:12bis}
\eta_t^N = \eta_0^N + \int_0^t L_s^* \eta_s^N ds + R^* W_t^N + h \int_0^t D^* g_{s}  d W_s^N(1) + h \int_0^t  \eta_s^N ( f) D^* g_s   ds +\bar R_t^{N } ,
\end{equation}
where, recalling \eqref{eq:en}, 
$\bar R_t^{N }  = R_t^N + h E^N + h \int_0^{\cdot}  \eta_s^N ( f) D^* ( \mu_s^N - g_s)    ds $ has already been shown to converge to $0$ in $ \W_0^{ - 4, p };$ that is, 
$$ \lim_{ N \to \infty} \E \sup_{ t \le T } \| \bar R_t^{N } \|_{ - 4, p } = 0.$$

Writing for short  $M^N = \int_0^\cdot D^* g_s d W^N_s (1) , $  \eqref{eq:12bis} implies that for all $t \geq 0, $ 
$$ F_k ( \eta^N, W^N, M^N ) (t)  = \bar R_t^{N} ( \psi_k) \; \mbox{ and }\; 
 \E ( \sup_{ t \le T } | \bar R_t^{N}|( \psi_k) |^2 )   \to 0$$
as $ N \to \infty.$ 
Therefore, we have convergence in probability,
$$ \sup_{t \le T } | F_k ( \eta^N, W^N, M^N ) (t) | \to 0 .$$
On the other hand, by the continuous mapping theorem, we have convergence in law $ F_k ( \eta^N, W^N, M^N )  \to F_k( \bar \eta, W, \int_0^\cdot D^* g_s d W_s ( 1) ) ;$ this allows to identify the limit which has to equal the zero-process.  

{\bf Step 2.3.} We now turn to the study of $G.$ Firstly, $ G$ is continuous at every point $( f^1, f^2, g^1) \in C( \R_+ , \W_0^{-4, p } \times \W_0^{-4, p } ) \times D ( \r_+, \R) .$ Indeed we only need to check the continuity of  
$$ g^1 \mapsto \left( t \mapsto \int_0^t g^1_s ds \right) $$
at every point $g^1 \in D ( \R_+, \R ) ,$ which follows from the basic properties of the Skorokhod topology. From \eqref{eq:utn} we have that 
$$ G ( \eta^N, W^N, U^{N, i }) = \bar R^{N, i } - \int_{ [0, \cdot] \times \r_+} U^{N, i }_{s-} \indiq_{\{ z \le f ( \bar X^i_{s-}) \}} \pi^i (ds, dz ) ,$$
and by \eqref{eq:lastpoint},   $\bar R^{N, i } $ converges to $ 0$ in probability, for the uniform convergence on finite time intervals, for any fixed $i.$ This implies that 
$$ G ( \bar \eta , W, \bar U^i ) = - \bar V^i.$$ 

It remains to identify 
\begin{equation}\label{eq:identify}
 \bar V^i =  \int_{[0, \cdot ] \times \r_+} \bar U^i_{s-} \indiq_{\{ z \le f ( \bar X^{i }_{s-} ) \}} \pi^i ( ds, dz) , 
\end{equation} 
for each $ 1 \le i \le n.$ 
 
In what follows, we write for short $ V^{N, i } = \int_{[0, \cdot ] \times \r_+}  U^{N,i }_{s-} \indiq_{\{ z \le f ( \bar X^{i }_{s-} ) \}} \pi^i ( ds, dz) .$ 
We already know that $ (U^{N, i }, V^{N, i } , \bar X^i )_{1 \le i \le n }  $ converges in law to $ (\bar U^i, \bar V^i, \bar X^i )_{1 \le i \le n } ,$ where once more, by abuse of notation, we use the same letter $ \bar X^i $ for the limit process of 
the third coordinate. Moreover, $ ( V^{N, i }, \bar X^i )_{1 \le i \le n }  $ is a semimartingale taking values in $\r^{2n}$ with characteristics 
$$ B^{N, i } = \left( \int_0^{\cdot } U^{N, i }_s f ( \bar X^i_s) ds , \int_0^\cdot [ - \alpha \bar X^i_s - \bar X^i_s f( \bar X^i_s) + g_s ( f) ] ds \right) , 1 \le i \le n , $$
$$ C^N = 0 , \nu^N (dt, dv, dx)  = \sum_{i=1}^n  f ( \bar X^i_{t-}) dt \left( \prod_{j=1, j \neq i}^n \delta_{(0, 0)} (d v^j , dx^j) \otimes \delta_{ (  - U^{N, i }_{t-} , - \bar X^i_{t-} ) }(dv^i,dx^i)\right) .$$ 
Clearly we have weak convergence $ ( V^{N, i }, \bar X^i ,  B^{N,i})_{1 \le i \le n}  \to ( \bar V^i, \bar X^i,  \bar B^i)_{1 \le i \le n } $ where 
$$ \bar B^i  =  \left( \int_0^{\cdot } \bar U^i_s f ( \bar X^i_s) ds , \int_0^\cdot [ - \alpha \bar X^i_s - \bar X^i_s f( \bar X^i_s) + g_s ( f) ] ds \right) , $$
by the continuity properties of the Skorokhod topology and since $ ( U^{N, i }, \bar X^i )_{1 \le i \le n} \to ( \bar U^i, \bar X^i)_{1 \le i \le n}.$

It is shown analogously that we have weak convergence $ g * \nu^N \to g * \bar \nu  ,$ for any continuous and bounded test function $g,$ where 
$$ \bar \nu = \sum_{i=1}^n  f ( \bar X^i_{t-}) dt \left( \prod_{j=1, j \neq i}^n \delta_{(0, 0)} (d v^j , dx^j) \otimes \delta_{ (  - \bar U^{ i }_{t-} , - \bar X^i_{t-} ) }(dv^i,dx^i)\right)  .$$ 
Then Jacod and Shiryaev \cite[Theorem 2.4 page 528]{js} implies that necessarily $ (\bar V^i , \bar X^i )_{1 \le i \le n} $ is a semimartingale with characteristics $ ( \bar B, 0, \bar \nu ) .$ Finally, the representation theorem \cite[Theorem III.2.26 page 157]{js} implies that there exist $n$ independent  Poisson random measures which, by abuse of notation, we still denote $ \pi^i (ds, dz ) ,$ having Lebesgue intensity, such that 
$$ \bar X^i_t = \bar X^i_0 - \alpha \int_0^t \bar X^i_s ds + \int_0^t h g_s ( f) ds  -\int_{[0, \cdot ] \times \r_+} \bar X^i_{s-} \indiq_{\{ z \le f ( \bar X^{i }_{s-} ) \}} \pi^i ( ds, dz)$$
and 
$$ \bar V^i_t = \int_{[0, \cdot ] \times \r_+} \bar U^i_{s-} \indiq_{\{ z \le f ( \bar X^{i }_{s-} ) \}} \pi^i ( ds, dz) .$$
This gives the desired identity \eqref{eq:identify} and thus finishes our proof.
\end{proof}

We close this section with the 
\begin{proof}[Proof of Theorem \ref{theo:main}]
Theorem \ref{theo:tight} implies the tightness of $ (\eta^N) $ and Theorem \ref{theo:tightU} the tightness of $ (U^N).$ Moreover, Theorem \ref{theo:char} implies that any limit $ (\bar \eta, \bar U) $ of $ (\eta^N, U^N) $ is solution of the system of differential equations \eqref{eq:baruoncemore}--\eqref{eq:baretabis}. Finally,  Theorem \ref{theo:uniqueness} implies pathwise uniqueness for this limit system, and the Yamada-Watanabe theorem allows to deduce weak uniqueness and thus the uniqueness of the limit law implying the weak convergence.
\end{proof}

\section{Appendix}
\subsection{Useful results on weighted Sobolev spaces}\label{sec:sobolev}
In what follows we collect the most important facts about weighted Sobolev spaces that are stated in Section 2.1 of \cite{fernandez}. First of all, obviously, for all $ k \le k', $ 
\begin{equation}\label{eq:goodtoknow}
 \| \cdot \|_{k, p } \le \| \cdot \|_{k', p }, \mbox{ implying that }  \| \cdot \|_{-k', p } \le \| \cdot \|_{-k, p }.
\end{equation} 
We also have that for all $ p \le p' , $ 
$$ \| \cdot \|_{k, p'} \le C \| \cdot \|_{k, p }, \mbox{ implying that }  \| \cdot \|_{-k , p } \le C \| \cdot \|_{-k, p' }. $$ 
Finally, $ C^{k, \alpha } \subset \W_0^{k, p } $ for any $ p > \alpha + \frac12.$

The following embeddings have been used throughout this paper. 
\begin{enumerate}
\item
{\bf Sobolev embedding}. There exists a constant $C$ such that for all $ m \geq 1, k \geq 0 $ and $ p \geq 0, $ 
$$ \| \psi\|_{C^{ k, p }} \le C \| \psi \|_{m+k, p }.$$
\item
{\bf Maurin's theorem}. The embedding $ \W_0^{m+k , p } \hookrightarrow  \W_0^{k , p + p'} $ is of Hilbert-Schmidt type for any $ m \geq 1, k \geq 0$ and $ p \geq 0, p' >1/2.$ This implies that the embedding is compact and that there exists a constant $C$ such that 
$$ \| \psi\|_{ k, p + p' } \le C \| \psi \|_{ k + m, p }.$$ 
\item
The dual embedding $  \W_0^{- k , p + p'  } \hookrightarrow  \W_0^{- (k+m) , p} $ is of Hilbert-Schmidt type. 
\end{enumerate}
We have also used several times that for any $ k, p \geq 0,$ there exists an orthonormal basis composed of $ C^\infty_c-$ functions $ (\psi_i)_i $ of $ \W_0^{ k, p } $ such that for any element $ w \in \W_0^{- k, p }, $ 
$$ \| w\|_{-k, p }^2 = \sum_i < w , \psi_i>^2.$$

\subsection{Proof of Lemma \ref{lemma:apriori}}
A straightforward adaptation of \cite[Prop.15]{evafournier} yields 
\begin{proposition}
Grant Assumption \ref{ass:1}. Then 
for all $t\geq 0$, all $i=1,\dots,N$,
\begin{eqnarray}
&& X^{N,i}_t \leq  X^{N,i}_0 +  3\bar X^N_0 + 4 h N^N_t, \label{tt1}\\
&&\frac{1}{N} \sum_{j = 1 }^N 
\int_0^t \int_0^\infty (h+X^{N,j}_{s-}) \indiq_{\{ z \le f ( X^{ N, j }_{s-}) \}} \pi^j (ds, dz )
\le 3 \bar X^N_0 + 4 h N_t^N, \label{tt2}
\end{eqnarray}
where $N^N_t:= N^{-1} \sum_{j = 1 }^N \int_{[0, t ] \times \r_+}  
\indiq_{\{ z \le f(2h) \}} \pi^j ( ds, dz ).$
\end{proposition}

\begin{proof} For the convenience of the reader we briefly sketch how to adapt the proof of \cite{evafournier} to the present frame. Taking the (empirical) mean of
\eqref{eq:dyn}, we find
\begin{equation}\label{eq:16}
 \bar X_t^N \le  \bar X_0^N + \frac{1}{N} \sum_{ i = 1 }^N \int_{[0, t ] \times \r_+}
\Big( h \frac{N-1}{N} - X^{N, i }_{s- } \Big) \indiq_{\{ z \le f ( X^{ N, i }_{s-}) \}} 
\pi^i (ds, dz ) 
\end{equation}
which implies
$$
\frac{1}{N} \sum_{ i = 1 }^N\int_{[0, t ] \times \r_+} ( X^{N, i }_{s- } - h  ) 
\indiq_{\{ z \le f ( X^{ N, i }_{s-}) \}} \pi^i (ds, dz ) \le \bar X_0^N .
$$
Using that $x-h\geq (x+h)/3 - (4/3)h \indiq_{ \{x\le 2h \} }$ for all $x \geq 0$ and that $f$ is non-decreasing,
we deduce that 
\begin{eqnarray*}
&&\frac{1}{N} \sum_{ i = 1 }^N\int_{[0, t ] \times \r_+}  
(h+X^{N,i}_{s-}) \indiq_{\{ z \le f ( X^{ N, i }_{s-}) \}} \pi^i (ds, dz )  \\
&\le&  3 \bar X_0^N 
+  \frac{4h }{N} \sum_{ i = 1 }^N \int_{[0, t ] \times \r_+}  \indiq_{ \{ X^{N, i }_{s- }  \le  2 h \} }   
\indiq_{\{ z \le f ( X^{ N, i }_{s-}) \}} \pi^i (ds, dz )\\
&\le &3 \bar X_0^N 
+  \frac{4h }{N} \sum_{ i = 1 }^N \int_{[0, t ] \times \r_+}
\indiq_{\{ z \le f ( 2h ) \}} \pi^i (ds, dz ).
\end{eqnarray*}
Now, for all $ 1 \le i \le N$, starting from \eqref{eq:dyn},
$$
 X_t^{N, i } \le X_0^{N, i } + \frac{h}{N} \sum_{ j = 1 }^N 
\int_{[0, t ] \times \r_+} \indiq_{ \{ z \le f ( X_{s-}^{N, j } ) \}} \pi^j ( ds, dz ) 
\le X_0^{N, i } 
+ 3 \bar X^N_0 + 4 h  N_t^N,
$$
which concludes. 
\end{proof}

The proof of \eqref{eq:aprioribound1} then follows from the fact that $ N_t^N= U/ N $ where $ U \sim Poiss ( N t f ( 2h ) ) $ and that $ g_0$ is of compact support. \eqref{eq:aprioribound2} is Proposition 14 of \cite{evafournier}. This finishes the proof of Lemma \ref{lemma:apriori}. $\qed$

\subsection{Useful properties of the limit process} 

\begin{lemma}\label{prop:lemma20}
For any $ p \geq 0, $ $g_t$ is continuous in $ \W_0^{-2, p },$ and for all $ t, t+h \le T, $ we have $\| g_{t+h } - g_t \|_{-2, p} \le C_T h.$  
\end{lemma}

\begin{proof}
We have for all $ \psi \in \W_0^{2, p }, $ 
$$ g_{t+h } (\psi ) - g_t ( \psi ) = \E [ \psi ( \bar X^1_{t+h}) - \psi ( \bar X^1_t) ] = \int_{t}^{t+h} \E L_s \psi ( \bar X^1_s) ds .$$
Using the Sobolev embedding theorem and  Lemma \ref{prop:5},  since $ |\bar X_s|  \le \bar C_T$ for all $s \le T, $ 
$$ | L_s \psi ( \bar X_s) | \le  \| L_s \psi \|_{C^{0 , p+\alpha} } ( 1 + |\bar X_s|^{p+\alpha} ) \le C_T \| L_s \psi \|_{1, p+\alpha } \le C_T \| \psi\|_{ 2, p} $$
implying that 
$$ |  g_{t+h } (\psi ) - g_t ( \psi ) | \le C_T h \| \psi\|_{2, p },$$
which concludes the proof. 
\end{proof}

We continue this section with the
\begin{proof}[Proof of Proposition \ref{prop:proprietegt}]
Firstly, since $g_0$ is of compact support,  $g_t$ is of compact support as well, for any fixed $t \geq 0.$ Moreover, 
similar arguments as those used in \cite[Theorem 2]{aaee} yield the explicit expression
\begin{equation}\label{eq:density1}
 g_t (y) = \frac{1}{h}e^{ \int_{\beta_t(y)} ^t ( \alpha - f ( \varphi_{\beta_t(y) , s} (0) ) ) ds} \indiq_{\{ y < \varphi_{0, t } (0) \}} + 
 e^{ \int_0^t  ( \alpha - f ( \varphi_{s,t }^{-1} (y)) ) ds } g_0 \circ \varphi_{0, t}^{-1} (y) \indiq_{\{ y \geq \varphi_{0, t } (0) \}},
\end{equation} 
with 
\begin{equation}\label{eq:density2}
 \varphi_{0, t}^{-1} (y)   = (y- \varphi_{0, t } (0) ) e^{ \alpha t }   \mbox{ and } \varphi_{s,t }^{-1} (y) = \varphi_{0, s } \circ \varphi_{0, t}^{-1}   (y) ,  \; y \geq \varphi_{0, t } (0) ,
\end{equation} 
and $ \beta_t (y) $ the unique real in  $]0 , t]$  satisfying 
\begin{equation}
\varphi_{ \beta_t(y ) , t } (0 ) = y,
\end{equation}
for any $ y < \varphi_{0, t } (0) .$ 

The above representation implies in particular that $ g_s (f) > 0 $ for all $ s \geq 0, $ since $ f (x) > 0 $ for all $ x >0.$ The function $ s\mapsto g_s (f) $ being continuous and strictly positive, we deduce that for any fixed $t>0,$ the function $ [0, t ] \ni s \mapsto \varphi_{s, t } (0) = \int_s^t e^{ - \alpha (t-u) } g_u (f) du $ is strictly decreasing and differentiable. As a consequence, its inverse function $ \beta_t $ is differentiable as well. Therefore, for any fixed $ t > 0, $ the Lebesgue density $g_t(y) $ is differentiable at every point $ y \neq \varphi_{0, t } (0) .$ The fact that 
$$
s \mapsto \int_0^\infty (1+ x^p ) |  g_s'(x) | dx \mbox{ is locally bounded}
$$
follows easily from this. 
\end{proof}

We finally give the 
\begin{proof}[Proof of Proposition \ref{prop:Pst}]

{\bf Step 1.} Before starting the proof, let us first mention that a simple change of variables formula implies that for any fixed $ s < t , $ the mapping $ \psi \mapsto [ x \mapsto \psi \circ \varphi_{s, t } (x) ] $ is continuous from $ \W_0^{6, p } \to \W_0^{6, p } ,$ where we recall that $ \varphi_{s, t} (x) = e^{ - \alpha ( t-s) } x + h \int_s^t g_u(f) du .$ Moreover we have 
$$ \| \psi \circ \varphi_{s, t }\|_{6,p} \le  C_T \| \psi \|_{6, p}, $$
for all $ s \le t \le T.$

{\bf Step 2.} Introduce now for any $ 0 \le s \le t$ and $ x \geq 0$  the process
\begin{equation}
 \bar Y_{s, t } ( x) =x + \int_s^t \left( h g_u (f) - \alpha \bar Y_{s, u } (x)) \right) du - \int_{]s, t ] \times \R_+} Y_{s, u-} (x) \indiq_{\{ z \le 1  \}} \pi^1 (du, dz ) ;
\end{equation} 
that is, $ \bar Y_{s, t } (x) $ follows the same dynamic as $Y_{s, t } (x), $ but jumps occur at constant rate $1.$ We still have the upper bound 
\begin{equation}\label{eq:boundbarY}
 \bar Y_{s, t } ( x) \le x  + \bar C_T, \mbox{ for all } s \le t \le T.
\end{equation} 
Let us write for short 
$$ \Pi_t =  \int_{[0, t ] \times \R_+}  \indiq_{\{ z \le 1  \}} \pi^1 (du, dz ) ,$$
that is, $(\Pi_t)_t $ is the Poisson process having intensity $1$ governing the jumps of $ \bar Y.$ Write $T_1 < T_2 < \ldots < T_n < \ldots $ for the successive jumps of $ (\Pi_t)_t.$ Then Girsanov's theorem for jump processes, see \cite{jacod75}, implies that 
$$ P_{s, t } \psi (x) = \E \psi ( Y_{s, t } (x) ) = \E \left( \psi ( \bar Y_{s, t } (x) ) \prod_{n :  T_n \in ]s, t ] } f ( \bar Y_{s, T_n -}(x) )e^{ - \int_s^t  [f( \bar Y_{s, u} ( x) )  - 1 ] du }\right)  .$$  
We notice that for all $ t < T_1 ( s) := \inf \{ T_n : T_n > s \} , $ $ \bar Y_{s, t } (x)  = \varphi_{s,t } ( x) .$ Therefore, 
\begin{multline*}
 P_{s, t } \psi (x)  = \psi ( \varphi_{s,t } (x) ) e^{ - \int_s^t (f ( \varphi_{s, u } ( x)) - 1 )  du }  \P ( t < T_1 (s) ) \\
 + e^{t-s} \E \left( f (\varphi_{s,T_1(s) }  (x) ) e^{ - \int_s^{T_1(s) }  f ( \varphi_{s, u } ( x)  ) du } R_{T_1(s) ,t} (\psi) ; t \geq T_1 ( s) \right) ,
\end{multline*} 
where 
$$  R_{T_1(s) ,t} (\psi)  = \psi ( \bar Y_{s, t} (x) ) \prod_{n :  T_n \in ]T_1(s) , t ] } f ( \bar Y_{s, T_n -}(x) )e^{ - \int_{{T_1(s) }}^t  f( \bar Y_{s, u} ( x) )   du}  .$$ 
Using the strong Markov property at time $T_1(s) $ and the fact that at time $ T_1 ( s), $ $ \bar Y_{s, T_1(s) } ( x) = 0 $ is reset to $0 $ and thus forgets its starting position $x$ at this time, we obtain 
\begin{multline*} 
\E \left( f (\varphi_{s,T_1(s) }  (x) ) e^{ - \int_s^{T_1(s) }  f ( \varphi_{s, u } ( x)  ) du } R_{T_1(s),t } (\psi)  ; t \geq T_1 ( s)  \right)\\
= 
\int_0^{t-s} e^{- v} f (\varphi_{s, s+v }  (x) ) e^{ - \int_s^{s+v}  f ( \varphi_{s, u } ( x)  ) du }  \bar R_{v, t} (\psi)   dv , 
\end{multline*} 
where
$$ \bar R_{v, t} (\psi)  = \E  \left(  \psi ( \bar Y_{s+v , t} ( 0 ) )  \prod_{n :  T_n \in ]s+ v , t ] } f ( \bar Y_{s+ v, T_n -}(0) )e^{ - \int_{s+v }^t  f( \bar Y_{s+v , u} ( 0) )   du} \right) .$$ 
Summarizing, we have 
\begin{eqnarray*}
 P_{s, t } \psi (x)&= &\psi ( \varphi_{s,t } (x) ) e^{ - \int_s^t f ( \varphi_{s, u } ( x)  ) du }  +e^{t-s}  \int_0^{t-s} e^{- v}f  (\varphi_{s,s+v }  (x) ) e^{ - \int_s^{s+v}  f ( \varphi_{s, u } ( x)  ) du } \bar R_{v, t } (\psi)  dv \\
 &:=&  P_{s, t }^1 \psi (x) + P_{s, t }^2 \psi ( x), 
\end{eqnarray*} 
where, using \eqref{eq:boundbarY},  
\begin{equation}\label{eq:goodbound1}
 \sup_{v \le t -s} |\bar R_{v,t } (\psi)| \le  C_t \| \psi \|_{1, p }. 
\end{equation} 

Clearly, $\psi \in C^6 $ implies that $ P_{s, t } \psi \in C^6 $ as well, since $f \in C^6.$ It is also clear at this stage that $ \psi \in C^6_c $ implies 
$P^1_{s, t } \psi \in C^6_c, $ having a support that depends on $ s$ and $t. $ 

{\bf Step 3.} 
We now investigate the dependence on $x$ of the first six derivatives of $ P_{s,t } \psi ( x) $ with respect to $x.$ Firstly, recalling \eqref{eq:goodbound1}, 
$$ \| P_{s, t } \psi  \|_{0, p } \le C_t \| \psi \|_{1, p }.$$
Let us now study the successive derivatives of $ P^1_{s, t} \psi.$ We have
$$ \frac{\partial}{\partial x} P_{s, t}^1 \psi ( x) = \left[ \psi '  ( \varphi_{s,t } (x) ) e^{ - \alpha ( t-s) } -  \psi( \varphi_{s,t } (x) )( \int_s^t f' ( \varphi_{s, u } (x) )e^{ - \alpha (u-s)}   du)  \right] e^{ - \int_s^t f ( \varphi_{s, u } ( x)  ) du } .$$
Since by Step 1, $ \| \psi \circ \varphi_{s, t }\|_{1,p} \le  C_T \| \psi \|_{1, p}, $ we only have to investigate the second term of the above expression. We use that $f$ is non-decreasing, that $ f' ( x) \le C (1 + x^\alpha ) ,$ $ f( x) \geq c_1 x^\beta \indiq_{\{ x \geq K\}}$ and that for all $ s \le u \le t \le T, $ 
$$ \varphi_{s, u } ( x) \geq e^{ - T} x \mbox{ such that  } f ( \varphi_{s, u} ( x) ) \geq f( e^{-T} x ) \geq c_1 e^{ - \beta T} x^\beta \indiq_{\{  x \geq e^T K\}}, $$
to deduce the upper bound 
$$ e^{ - \int_s^t f( \varphi(_{s, u } (x) ) du }  \le e^{ - c_1 e^{- \beta T} (t-s) x^\beta  }  \indiq_{\{  x \geq e^T K\}} + \indiq_{\{  x < e^T K\}},$$
implying 
$$ \sup_x  \left(\left(\int_s^t  | f' ( \varphi_{s, u } (x) )| e^{ - \alpha (u-s)}  du\right)  \; e^{ - \int_s^t f ( \varphi_{s, u } ( x)  ) du } \right) = C_T < \infty  $$ 
and thus
$$ \| P_{s, t}^1 \psi \|_{1, p } \le C_T \| \psi\|_{1, p}.$$ 
Similar arguments give 
$  \| P_{s, t}^1 \psi \|_{6, p } \le C_T \| \psi\|_{6, p}.$

Finally, the same arguments as above give that for all $ s \le t, $ 
$$ x \mapsto  \int_0^{t-s} e^{- v}f  (\varphi_{s,s+v }  (x) ) e^{ - \int_s^{s+v}  (f ( \varphi_{s, u } ( x)  ) du } =: F_{s, t } ( x) \in C^6_b $$ 
and for all $ \gamma > 0 $ and all $ 0 \le k \le 6, $  
$$\lim_{x \to \infty } x^\gamma | F_{s, t }^{(k)} ( x) |  = 0$$ 
such that 
$$  \| P_{s, t}^2 \psi \|_{6, p } \le C_t \| \psi\|_{6, p} \mbox{ and even }  \| x^\gamma P_{s, t}^2 \psi \|_{6, p } \le C_t ( \gamma)  \| \psi\|_{6, p} ,$$
where $ x^\gamma P_{s, t}^2 \psi  $ denotes the function $ x \mapsto x^\gamma P_{s, t } \psi ( x) .$ This concludes the proof. 
\end{proof}

\textbf{Acknowledgments.}
This work has been conducted as part of  the  FAPESP project Research, Innovation and Dissemination Center for Neuromathematics(grant 2013/07699-0) and of the ANR project ANR-19-CE40-0024.


\begin{thebibliography}{99}
\bibitem{bill}
{\sc Billingsley, P. }
\newblock Convergence of probability measures. 
\newblock New York, London, Sydney, Toronto: John Wiley $\&$ Sons 1968.

\bibitem{chevallier}
{\sc Chevallier, Julien.}
\newblock
Fluctuations for mean-field interacting age-dependent Hawkes processes. 
\newblock {\it Electron. J. Probab. } 22 (42), 1-- 49, 2017.

\bibitem{duartechevallier}
{\sc Chevallier, J., Duarte, A., L\"ocherbach, E., Ost, G.}
\newblock Mean field limits for nonlinear spatially extended Hawkes processes with exponential memory kernels.
\newblock {\em Stoch. Proc. Appl.} 129,  1--27, 2019.

\bibitem{tanre1}
{\sc Cormier, Q., Tanr\'e, E., Veltz, R.}
\newblock Long time behavior of a mean-field model of interacting neurons.
\newblock {\em Stoch. Proc. Appl}.  130, 2553-2595, 2020.

\bibitem{tanre2}
{\sc Cormier, Q., Tanr\'e, E., Veltz, R.}
\newblock  Hopf bifurcation in a Mean-Field model of spiking neurons
\newblock {\em Electronic Journal of Probability} 26, 1--40, 2021.

\bibitem{aaee}
{\sc De Masi, A., Galves, A., L\"ocherbach, E., Presutti, E.}
\newblock
Hydrodynamic limit for interacting neurons. 
\newblock {\it J. Stat. Phys.} 158, 866--902, 2015.


\bibitem{duarte}
{\sc Duarte, A., Rodr\'iguez, A. A., Ost, G.}
\newblock Hydrodynamic Limit for Spatially Structured Interacting Neurons.
\newblock {\em Journal of Statistical Physics} 161, 1163--1202, 2015.

\bibitem{ferland}
{\sc Ferland, R., Fernique, X. and Giroux, G.}
\newblock  Compactness of the fluctuations associated with some generalized nonlinear Boltzmann equations.
\newblock {\em Canadian Journal of Mathematics} 44, 1192--1205, 1992.

\bibitem{fernandez}
{\sc Fernandez, B.,  M\'el\'eard, S.}
\newblock A Hilbertian approach for fluctuations on the McKean-Vlasov model.
\newblock {\em Stoch. Proc. Appl.} 71, 33--53, 1997.


\bibitem{evafournier}
{\sc Fournier, N, L\"ocherbach, E.}
\newblock On a toy model of interacting neurons.
\newblock{\it  Ann. Inst. H. Poincar\'e Probab. Statist.} 52, 1844--1876, 2016.

\bibitem{jacod75}
{\sc Jacod, J.}
\newblock Multivariate point processes: predictable projection, Radon-Nikodym derivates, representation of martingales. 
\newblock {\it Z. Wahrscheinlichkeitstheor. Verw. Geb.}  31 (1975), 235-253.

\bibitem{js}
{\sc Jacod, J., Shiryaev, A.N.}
\newblock Limit theorems for stochastic processes.
\newblock Second edition, Springer-Verlag, Berlin, 2003.

\bibitem{joffe}
{\sc Joffe, A., M\'etivier, M.}
\newblock Weak convergence of sequences of semimartingales with applications to multitype branching processes.
\newblock {\it Advances in Appl. Probability}, 20--65, 1986.

\bibitem{evapierre}
{\sc L\"ocherbach, E., Monmarch\'e, P.}
\newblock Metastability for systems of interacting neurons.
\newblock {\it  Ann. Inst. H. Poincar\'e Probab. Statist.} To appear, 2022.

\bibitem{robert-touboul}
{\sc Robert, P.,  Touboul, J.}
\newblock On the dynamics of random neuronal networks.
\newblock {\em J. Stat. Phys. 165,} (2016), 545--584.


\bibitem{sznitman}
{\sc Sznitman, A.-S.}
\newblock Topics in propagation of chaos.
\newblock In {\em \'{E}cole d'\'{E}t\'e de {P}robabilit\'es de {S}aint-{F}lour
{XIX}---1989}, vol.~1464 of {\em Lecture Notes in Math.} Springer, Berlin,
1991, 165--251.
\end{thebibliography}
\end{document}